\documentclass[a4paper]{amsart}

\usepackage{hyperref}
\usepackage[all]{xy}
\usepackage[usenames]{color}
\usepackage{lscape}
\usepackage{amsmath}
\usepackage{amssymb}

\definecolor{mint}{rgb}{0.13, 0.67, 0.54}
\def\white{\color{white}}
\def\red{\color{red}}

\newtheorem{proposition}{Proposition}
  \newtheorem{lemma}[proposition]{Lemma}
\newtheorem*{lemma*}{Lemma}
  \newtheorem{theorem}[proposition]{Theorem}
\theoremstyle{definition}
  \newtheorem{definition}[proposition]{Definition}
  \newtheorem*{definition*}{Definition}
  \newtheorem{example}[proposition]{Example}
  \newtheorem*{example*}{Example}

\def\ox{\otimes}

\def\stac#1{\raisebox{-4pt}{${\stackrel{\ox}{{}_{#1}}}$}}
\def\psiinv{\phi}
\def\flip{\mathsf{\,tw\,}}
\def\tensor#1{\otimes_{#1}}

\begin{document}
\title{On the double crossed product of weak Hopf algebras}
\author{Gabriella B\"ohm} 
\address{Wigner Research Centre for Physics, Budapest,
H-1525 Budapest 114, P.O.B.\ 49, Hungary}
\email{bohm.gabriella@wigner.mta.hu}
\author{Jos\'e G\'omez-Torrecillas}
\address{Departamento de \'Algebra,  Universidad de Granada,
           E-18071 Granada, Spain}
\email{gomezj@ugr.es}
\begin{abstract}
Given a weak distributive law between algebras underlying two weak bialgebras,
we present  sufficient conditions under which the corresponding
weak wreath product algebra becomes a weak bialgebra with respect to the
tensor product coalgebra structure. When the weak bialgebras are weak Hopf
algebras,  then the same conditions are shown to imply that the weak
wreath product becomes a weak Hopf algebra, too. Our sufficient
conditions are capable to describe most known examples, (in
particular the Drinfel'd double of a weak Hopf algebra) . 
\end{abstract}
\date{}
\maketitle

\thispagestyle{empty}
\section*{Introduction}
Bialgebras can be regarded as algebras in the monoidal category
$\mathsf{coalg}$ of coalgebras. Hence a distributive law in $\mathsf{coalg}$
-- that is, a distributive law between the underlying algebras of two
bialgebras, which is also a homomorphism of coalgebras -- induces a
wreath product bialgebra; with the multiplication twisted by the distributive
law and the tensor product comultiplication. This is known as Majid's double
crossed product construction \cite{Majid}.  

More generally, applying the construction of {\em weak} wreath product in any
monoidal category  with split idempotents \cite{BGT}, one can take
a {\em weak} distributive law in $\mathsf{coalg}$ -- that is, a weak
distributive law between algebras underlying bialgebras, which is a coalgebra
homomorphism. It yields an algebra in $\mathsf{coalg}$; that is,
a bialgebra again.  (The simplest kind of examples is given as follows: Let $A$
be a bialgebra over a commutative ring $k$ in which there exists a grouplike
$e \in A$ such that $ea=eae$ for every $a \in A$. Then the map $\Psi : A \cong
A \otimes k \to k \otimes A \cong A$ defined by $\Psi (a) = ea$ for $a \in A$
is a weak distributive law in $\mathsf{coalg}$. The corresponding  weak wreath
product is isomorphic to $eA$, which is a bialgebra with unit $e$. A minimal
proper example of this construction is $A = kS$, the monoid (bi)algebra of the
monoid $S = \{ e, 1 \}$ with multiplication $e^2 = e 1 = 1 e = e, 1^2 = 1$.)  

The aim of this paper is to study double crossed products of {\em weak}
bialgebras. By this we mean weak bialgebras which -- {\em as algebras} --
arise as a weak wreath product of two weak bialgebras, and whose {\em
  coalgebra} structure comes from the tensor product coalgebra. (Note that
this does not fit the construction in \cite{santi}, where both the algebra and
coalgebra structures are twisted by weak distributive laws of a {\em common}
image.) 

The difficulty of the problem comes from the fact that no description of weak
bialgebras as algebras in some well-chosen monoidal category is known. Hence
there is no evident notion of (weak) wreath product of weak bialgebras. On the
other hand, many examples of double crossed product weak bialgebras (in the
above sense) are known. 

Our strategy is to take a weak distributive law between algebras underlying
weak bialgebras. Then we look for sufficient conditions under which the
corresponding weak wreath product algebra becomes a weak bialgebra with
respect to the tensor product coalgebra structure.  
The conditions we present are only sufficient for the desired weak bialgebra
to exist. Although it is possible to give the sufficient and necessary
conditions, they are technically involved and so do not seem to be usable in
practice. Our sufficient conditions, however, have a simple form and they are
capable to describe the known examples (in particular the Drinfel'd double of a
weak Hopf algebra \cite{Bohm:DH,Nenciu,CWY}). 
\bigskip

A {\em weak bialgebra} \cite{BNSz:WHAI} over a commutative ring $k$ is a
$k$-module $H$ equipped with a $k$-algebra structure $(\mu,\eta)$ and a
$k$-coalgebra structure $(\Delta, \epsilon)$, subject to the following axioms, 
$$
\xymatrix{
H^{\otimes 2}\ar[r]^-\mu\ar[d]^-{\Delta\otimes \Delta}&
H\ar[dd]_-\Delta&
k\ar[rr]^-{\eta\ox \eta}\ar[d]^-{\eta\ox \eta} \ar@/^.5pc/[rd]^-\eta&&
H^{\otimes 2}\ar[d]_-{\Delta\otimes \Delta}&
H^{\otimes 3}\ar[rr]^-{H\otimes \Delta^{op}\otimes H}
\ar[d]^(.7){\!H\scalebox{.5}{$\otimes$}\Delta\scalebox{.5}{$\otimes$} H}
\ar@/^.5pc/[rd]^-{\mu^2}&&
H^{\otimes 4}\ar[d]_-{\mu \ox \mu}\\
H^{\otimes 4}\ar[d]^-{H\otimes \flip\otimes H}&&
H^{\otimes 2}\ar[d]^-{\Delta\otimes \Delta}&
H\ar@/_.5pc/[rd]_-{\Delta^2}&
H^{\otimes 4}\ar[d]_(.3){H{\scalebox{.5}{$\otimes$}}
\mu\scalebox{.5}{$\otimes$}  H}&
H^{\otimes 4}\ar[d]^-{\mu \ox \mu}&
H\ar@/_.5pc/[rd]_-\epsilon&
H^{\otimes 2}\ar[d]_-{\epsilon\otimes \epsilon}\\
H^{\otimes 4}\ar[r]_-{\mu \ox \mu}&
H^{\otimes 2}&
H^{\otimes 4}\ar[rr]_-{H\otimes \mu^{op}\otimes H}&&
H^{\otimes 3}&
H^{\otimes 2}\ar[rr]_-{\epsilon\otimes \epsilon}&&
k}
$$
where $\flip:H\ox H\to H\ox H$ is the twist map $a\ox b\mapsto b\ox a$,
$\mu^{op}=\mu\flip$ and $\Delta^{op}=\flip\Delta$. 
Note, however, that unitality of the comultiplication; that is, $\Delta \eta= \eta\ox \eta$ is {\em not} required. On elements $a,b,c$ of
$H$, the axioms take the following form. 
\begin{eqnarray}
\Delta(ab)&=&\Delta(a)\Delta(b)\label{eq:Delta_mu}\\
(\Delta(1)\ox 1)(1\ox \Delta(1))=&\Delta^2(1)&=
(1\ox \Delta(1))(\Delta(1)\ox 1)\label{eq:Delta(1)}\\
\epsilon(ab_1)\epsilon(b_2c)=&\epsilon(abc)&=
\epsilon(ab_2)\epsilon(b_1c).\label{eq:epsilon_mu}
\end{eqnarray}
Here we are using a simplified version of Heynemann-Sweedler's notation:
$\Delta(a) = a_1 \otimes a_2$, implicit summation understood.
The (in fact idempotent) maps $H\to H$,
$$
\sqcap^R:a\mapsto 1_1\epsilon(a1_2),\quad 
\sqcap^L:a\mapsto \epsilon(1_1a)1_2,\quad 
\overline\sqcap^R:a\mapsto 1_1\epsilon(1_2a),\quad 
\overline\sqcap^L:a\mapsto \epsilon(a1_1)1_2,
$$
play an important role. Recall from \cite{BNSz:WHAI} that they obey, for
example, the following identities. For all $a,b\in H$, 
\begin{equation}\label{eq:wba_id} 
1_1b\ox 1_2 = b_1\ox \sqcap^L(b_2) \qquad \textrm{and}\qquad
\epsilon(a_1b)a_2=a \sqcap^L(b).
\end{equation}
By \cite[Lemma 1.2]{BCJ}, axioms \eqref{eq:epsilon_mu}
can be equivalently replaced by 
\begin{equation}\label{eq:alternative_ax}
\epsilon(ab_2)b_1=\sqcap^R(a)b
\quad\textrm{and}\quad 
\epsilon(ab_1)b_2=\overline{\sqcap}^L(a)b.
\end{equation}
A weak bialgebra is said to be a {\em weak Hopf algebra} if there exists a
linear map $S:H \to H$ rendering commutative the following diagrams. 
$$
\xymatrix@R=15pt @C=45pt {
H\ar[r]^-{H\otimes \Delta\circ \eta}\ar[d]^-\Delta&
H^{\otimes 3}\ar[dd]_(.4){\epsilon\circ \mu^{op}\otimes H}&
H\ar[r]^-{\Delta\circ \eta\otimes H}\ar[d]^-\Delta&
H^{\otimes 3}\ar[dd]_(.4){H\otimes \epsilon\circ \mu^{op}}&
H\ar[r]^-S\ar[dd]^-{\Delta^2}&
H\\
H^{\otimes 2}\ar[d]^-{H\otimes S}&&
H^{\otimes 2}\ar[d]^-{S\otimes H}\\
H^{\otimes 2}\ar[r]_-\mu&
H&
H^{\otimes 2}\ar[r]_-\mu&
H&
H^{\otimes 3}\ar[r]_-{S\otimes H \otimes S\ }&
H^{\otimes 3}\ar[uu]^-{\mu^2}}
$$
On elements $a\in H$,
\begin{equation}\label{eq:antip_ax}
a_1S(a_2)=\sqcap^L(a)\qquad\qquad\qquad\qquad
S(a_1)a_2=\sqcap^R(a)\qquad\qquad\qquad\qquad
S(a_1)a_2S(a_3)=S(a).
\end{equation}
For more on weak bialgebras and weak Hopf algebras, we refer to
\cite{BNSz:WHAI}.  
\bigskip

The paper is organized as follows. In Section \ref{sec:d_cr} we define weakly
invertible weakly comonoidal weak distributive laws between weak
bialgebras. We prove that they induce double crossed product weak
bialgebras. In Section \ref{sec:antip} we prove that starting from two weak
{\em Hopf} algebras, under the same conditions as in Section
\ref{sec:d_cr} we obtain a double crossed product weak {\em Hopf}
algebra. In final Section \ref{sec:ex}, we collect a number of examples of
double crossed product weak bialgebras and weak Hopf algebras, and show how
they fit our theory. For the convenience of the reader, we collected in an
Appendix the (sometimes big) diagrams that are used in the proofs of Sections
\ref{sec:d_cr} and \ref{sec:antip}.  
\bigskip

{\bf Acknowledgements.} GB thanks the members of Departamento de
\'Algebra at Universidad de Granada for a generous invitation and for a very
warm  hospitality experienced during her visit in February of 2012. 
Partial financial support from the Hungarian Scientific Research Fund OTKA,
grant no. K68195, and from the Ministerio de Ciencia e Innovaci\'on and FEDER,
grant MTM2010-20940-C02-01 is gratefully acknowledged.

\section{Double crossed product of weak bialgebras}\label{sec:d_cr}

The following definition, and the corresponding product construction of
algebras, appeared in \cite{CDG} and \cite{Str:Wdl}. 
\begin{definition}
For algebras $(A,\mu,\eta)$ and $(B,\mu,\eta)$ over a commutative ring $k$, a
{\em weak distributive law} is a $k$-module map $\psi:A \ox B\to B \ox A$
subject to the following conditions. 
\begin{equation}\label{eq:wdl}
\begin{array}{ll}
\psi(\mu\ox B)=(B\ox \mu)(\psi\ox A)(A\ox \psi)\quad
&\psi(\eta \ox B)=(\mu\ox A)(B\ox \psi)(B \ox \eta \ox \eta)\\
\psi(A\ox \mu)=(\mu\ox A)(B\ox \psi)(\psi\ox B)\quad
&\psi(A\ox \eta)=(B\ox \mu)(\psi\ox A)(\eta \ox \eta \ox A).
\end{array}
\end{equation}
The two conditions on the right can be replaced equivalently by 
$$
(B\ox \mu)(\psi\ox A)(\eta \ox  B \ox A)=
(\mu\ox A)(B\ox \psi)(B \ox  A \ox \eta).
$$
\end{definition}

Consider a weak distributive law $\psi:A\ox B \to B\ox A$ between algebras $A$
and $B$ over a commutative ring $k$.
\begin{definition}
We say that a weak distributive law $\psiinv:B\ox A\to A\ox B$ is the
{\em weak inverse} of $\psi$ if
\begin{equation}\label{eq:inverse}
\psi\,\psiinv=(\mu\ox A)\, (B\ox \psi)\, (B\ox A\ox \eta)
\quad \textrm{and}\quad 
\psiinv\,\psi=(\mu\ox B)\, (A\ox \psiinv)\, (A\ox B\ox \eta).
\end{equation}
\end{definition}
The weak inverse of $\psi$ is not unique and it may not exist. If it exists
then it obeys
$$
\psiinv\,\psi\,\psiinv=
(\mu\ox B)\, (A\ox \psiinv)\, (A\ox B\ox \eta)\,\psiinv=
\psiinv\quad \textrm{and}\quad
\psi\,\psiinv\,\psi=\psi.
$$
If $\psi$ is a proper distributive law then it has a weak inverse which is
also a proper distributive law if and only if $\psi$ is a bijective map. Such
an inverse is clearly unique (if it exists).

Recall that for coalgebras $(A,\Delta,\epsilon)$ and
$(B,\Delta,\epsilon)$ over a commutative ring $k$, the $k$-module tensor
product $A\ox B$ is also a coalgebra via 
$$
\Delta_{A\ox B} = (A\ox \flip\ox B)(\Delta \ox \Delta)
\qquad \textrm{and}\qquad 
\epsilon_{A\ox B} = \epsilon \ox \epsilon.
$$
Symmetrically, $B\ox A$ is a coalgebra too. 

Let both $A$ and $B$ carry algebra structures $(\mu,\eta)$ and coalgebra
structures $(\Delta,\epsilon)$ over a commutative ring $k$ (no
compatibility is assumed at this stage). 
Consider a mutually weak inverse pair of weak distributive laws
$(\psi:A\ox B\to B\ox A, \psiinv:B\ox A\to A\ox B)$. 

\begin{definition} \label{def:weak_comon}
We say that the pair $(\psi,\psiinv)$ is {\em weakly comonoidal} if the
following equalities hold. 
\begin{eqnarray}
(\psi\,\psiinv\ox B\ox A)\, \Delta_{B\ox A}\,\psi&=
(\psi\ox\psi)\, \Delta_{A\ox B}=&
(B\ox A\ox \psi\,\psiinv)\, \Delta_{B\ox A}\,\psi
\label{eq:psibar}\\
(\psiinv\,\psi\ox A\ox B)\, \Delta_{A\ox B}\,\psiinv&=
(\psiinv\ox\psiinv)\, \Delta_{B\ox A}=&
(A\ox B\ox \psiinv\,\psi)\, \Delta_{A\ox B}\,\psiinv
\label{eq:psi}\\
&\epsilon_{B\ox A}\, \psi=\epsilon_{A\ox B}\, \psiinv\, \psi.&
\label{eq:counit}
\end{eqnarray}
\end{definition}
Clearly, \eqref{eq:counit} is equivalent to $\epsilon_{A\ox B}\,
\psiinv=\epsilon_{B\ox A}\, \psi\, \psiinv$.  
If $\psi$ is an invertible distributive law, then it is weakly comonoidal if
and only if it is a coalgebra homomorphism.  

\begin{theorem}\label{thm:product_wba}
Consider weak bialgebras $A$ and $B$ over a commutative ring $k$. For a
weakly comonoidal mutually weak inverse pair of weak distributive
laws $(\psi:A\ox B\to B\ox A, \psiinv:B\ox A\to A\ox B)$, the following
hold. 
\begin{itemize}
\item[{(1)}] The image $B\ox_\psi A$ of the idempotent map 
$\psi\psiinv:B\ox A \to B \ox A$ is an algebra via the weak wreath product
construction. 
\item[{(2)}] $B\ox_\psi A$ is a coalgebra via the comultiplication
$$
\xymatrix{
B\ox_\psi A\ \ar@{>->}[r]&B\ox A\ar[r]^-{\Delta_{B\ox A}}&
B\ox A\ox B\ox A\ar@{->>}[r]&(B\ox_\psi A)\ox (B\ox_\psi A)
}
$$
and the counit
$$
\xymatrix{
B\ox_\psi A \ \ar@{>->}[r]&B\ox A\ar[r]^-{\epsilon_{B\ox A}}&k,
}
$$
defined in terms of the tensor product coalgebra $B\ox A$.
\item[{(3)}] The algebra in part (1) and the coalgebra in part (2) 
constitute a weak bialgebra.
\end{itemize}
\end{theorem}

\begin{proof}
Part (1) follows by \cite[Theorem 2.4]{Str:Wdl} (see also \cite[Theorem
1.6]{BGT}). Let us just recall that the multiplication on $B
\ox_\Psi A$ is given by 
$$
\xymatrix @C=15pt{
B \ox_\Psi A \ox B \ox_\Psi A \ \ar@{>->}[r]&
B \ox A \ox B \ox A \ar^-{B \ox \Psi \ox A }[rr]&&
B \ox B \ox A \ox A \ar^-{\mu \ox \mu}[r]& 
B \ox A \ar@{->>}[r] & B \ox_\Psi A}, 
$$ 
and the unit is  
$$ 
\big(
\xymatrix{
k \ar^-{\eta \ox \eta}[r]& 
A \ox B \ar^{\Psi}[r]& 
B \ox A \ar@{->>}[r] & B \ox_\Psi A }
\big) \stackrel{\eqref{eq:inverse}}=\big(
\xymatrix{
k \ar^-{\eta \ox \eta}[r]& 
B \ox A \ar@{->>}[r] & B \ox_\Psi A }
\big).
$$

As for (2) is concerned,  it follows from \eqref{eq:psibar} and
  \eqref{eq:psi} that 
\begin{equation}\label{eq:quotient}
(\psi\,\psiinv \ox\psi\,\psiinv)\,\Delta_{B\ox A}\, \psi\,\psiinv=
(\psi \ox \psi)(\psiinv\psi\ox A \ox B)\Delta_{A\ox B}\psiinv=
(\psi\,\psiinv \ox\psi\,\psiinv)\,\Delta_{B\ox A}.
\end{equation}
(In other words, the epimorphism $B\ox A \to B\ox_\psi
A$ is comultiplicative.) Combining this with the coassociativity of
$\Delta_{B\ox A}$, we conclude on the coassociativity of the
comultiplication on $B\ox_\psi A$. 

Counitality follows by commutativity of the following diagram
$$
\xymatrix @R=15pt @C=50pt{
B\stac \psi A \ \ar@{>->}[r]\ar@{=}@/_7.3pc/[rrrddd]&
B\ox A \ar[r]^-{\Delta_{B\ox A}}\ar[d]^-{\psiinv}
\ar@{}[rdd]|-{\red\eqref{eq:psi}}&
(B\ox A)^{\ox 2}\ar[dd]^-{\psi\psiinv\ox \psiinv}
\ar[rd]^-{(\psi\psiinv)^{\ox 2}}\ar@{->>}[r]
\ar@{}[rdd]|-{\red\eqref{eq:counit}}&
(B\stac \psi A)^{\ox 2} \ar@{>->}[d]\\
&A\ox B\ar[d]^-{\Delta_{A\ox B}}\ar@/_3pc/@{=}[dd]&&
(B\ox A)^{\ox 2}\ar[d]_(.3){B\ox A \ox \epsilon_{B\ox A}}\\
&(A\ox B)^{\ox 2} \ar[d]^-{A\ox B \ox \epsilon_{A\ox B}} 
\ar[r]^-{\psi\ox A \ox B}&
B\ox A\ox A\ox B\ar[r]^-{B\ox A \ox \epsilon_{A\ox B}}&
B\ox A \ar@{->>}[d]\\
&A\ox B\ar[rru]_-{\psi}&&
B\stac \psi A
}
$$
\vskip 1cm

\noindent
and a symmetrical one on the other side.

Let us turn to part (3). By \eqref{eq:inverse} and the axioms of a weak
distributive law, 
$$
\psi\,\psiinv=(\mu \ox A)\, (B\ox \psi)\, (B\ox A\ox \eta)=
(B\ox \mu)\, (\psi\ox A)\, (\eta\ox B\ox A).
$$
Hence associativity of $\mu$ implies that $\psi\,\psiinv$ is a $B$-$A$
bimodule map in the sense that 
\begin{equation}\label{eq:bilin}
(\mu\ox\mu)\,(B\ox \psi\,\psiinv\ox A)=
\psi\,\psiinv\,(\mu\ox\mu).
\end{equation}
This implies
\begin{equation}\label{eq:onleft}
(\mu\ox\mu)\,(B\ox \psi \ox A)=
\psi\,\psiinv\,(\mu\ox\mu)\,(B\ox \psi \ox A).
\end{equation}
Furthermore, by \eqref{eq:inverse} and the axioms of a weak distributive
law,
\begin{eqnarray}\label{eq:onright}
(\mu \ox A)\, (B\ox \psi)\, (\psi\,\psiinv\ox B)&=&
(\mu \ox A)\, (B\ox \psi)\,(\mu\ox A\ox B)\, (B\ox \psi\ox B)\, 
(B\ox A\ox \eta\ox B)\nonumber\\
&=& (\mu \ox A)\,(\mu\ox B \ox A)\,(B\ox B\ox \psi)\, (B\ox \psi\ox B)\, 
(B\ox A\ox \eta\ox B)\nonumber\\
&=&(\mu \ox A)\,(B\ox \mu \ox A)\,(B\ox B\ox \psi)\, (B\ox \psi\ox B)\, 
(B\ox A\ox \eta\ox B)\\
&=&(\mu \ox A)\, (B\ox \psi)\, (B\ox A\ox \mu)\, (B\ox A\ox \eta\ox B)= 
(\mu \ox A)\, (B\ox \psi),\nonumber\\
(B\ox \mu)\, (\psi\ox A)\, (A\ox \psi\psiinv)&=&
(B\ox \mu)\, (\psi\ox A).\nonumber
\end{eqnarray}
Then the compatibility between the multiplication and the comultiplication
follows by commutativity of 
$$
\xymatrix @C=37pt {
(B\stac \psi A)^{\ox 2}\ \ar@{>->}[d]\ar@{>->}[r]&
(B\ox A)^{\ox 2}\ar[r]^-{B\ox \psi\ox A}
\ar@{}[rrrd]|-{\red \eqref{eq:onleft}\eqref{eq:onright}}&
B^{\ox 2}\ox A^{\ox 2}\ar[r]^-{\mu\ox \mu}&
B\ox A \ar@{->>}[r]\ar[rd]^-{\psi\psiinv}&
\ \ \ B\stac \psi A \ar@{>->}[d]\\
(B\ox A)^{\ox 2}\ar[dd]_-{\Delta_{(B\ox A)^{\ox 2}}}
\ar@{}[rdd]|(.28){\red\eqref{eq:quotient}}&
(B\ox A)^{\ox 2}\ar[l]_-{(\psi \psiinv)^{\ox 2}} 
\ar[u]_-{(\psi \psiinv)^{\ox 2}}\ar@{->>}[lu]\ar[r]^-{B\ox \psi\ox A}
\ar[dd]^-{\Delta_{(B\ox A)^{\ox 2}}}
\ar@{}[rdd]|(.29){\red \eqref{eq:psibar}}&
B^{\ox 2}\ox A^{\ox 2}
\ar[d]^-{\Delta_{B^{\ox 2}\ox A^{\ox 2}}}\ar[rr]^-{\mu\ox\mu}
\ar@{}[rrd]|-{\red \eqref{eq:Delta_mu}}&&
B\ox A\ar[dd]^-{\Delta_{B\ox A}}\\
&&(B^{\ox 2}\ox A^{\ox 2})^{\ox 2}
\ar[d]^-{(B\ox \psi\psiinv\ox A)^{\ox 2}}
\ar[rrd]^-{(\mu\ox\mu)^{\ox 2}}
\ar@{}[rd]|(.6){\red\eqref{eq:bilin}}&&\\
(B\ox A)^{\ox 4}\ar@{->>}[d]\ar[rd]^-{(\psi \psiinv)^{\ox 4}}&
(B\ox A)^{\ox 4}\ar[d]^-{(\psi \psiinv)^{\ox 4}}
\ar[r]^-{\raisebox{5pt}{${}_{\,\,(B\ox \psi\ox A)^{\ox 2}}$}}
\ar@{}[rrd]|-{\red \eqref{eq:onleft}\eqref{eq:onright}} &
(B^{\ox 2}\ox A^{\ox 2})^{\ox 2}\ar[r]^-{(\mu\ox\mu)^{\ox 2}}&
(B\ox A)^{\ox 2}&
(B\ox A)^{\ox 2}\ar[l]^-{(\psi \psiinv)^{\ox 2}}\ar@{->>}[d]\\
(B\stac \psi A)^{\ox 4}\ \ar@{>->}[r]&
(B\ox A)^{\ox 4}\ar[r]_-{\raisebox{-5pt}{${}_{(B\ox \psi\ox A)^{\ox 2}}$}}&
(B^{\ox 2}\ox A^{\ox 2})^{\ox 2}\ar[r]_-{(\mu\ox \mu)^{\ox 2}}&
(B\ox A)^{\ox 2}\ar@{->>}[r]
\ar[u]_-{(\psi \psiinv)^{\ox 2}}&
(B\stac \psi A)^{\ox 2}\ar@{>->}[lu]
}
$$ 
The label ``\eqref{eq:Delta_mu}'' means that the square commutes by the weak
bialgebra axiom \eqref{eq:Delta_mu} holding true both in $A$ and $B$. 

The comultiplication takes the unit of $B\ox_\psi A$ to 
\begin{eqnarray}\label{eq:delta1}
&&\big(\xymatrix{
k\ar[r]^-{\eta\ox \eta}&A\ox B\ar[r]^-{\psi}&B\ox A\ar@{->>}[r]&
B\ox_\psi A\ \ar@{>->}[r]&B\ox A \ar[r]^-{\Delta_{B\ox A}}&
(B\ox A)^{\ox 2} \ar@{->>}[r]&
(B\ox_\psi A)^{\ox 2}
}\big)=\nonumber\\
&&\big(\xymatrix{
k\ar[r]^-{\eta\ox \eta}&A\ox B\ar[r]^-{\psi}&B\ox A \ar[r]^-{\Delta_{B\ox A}}& 
(B\ox A)^{\ox 2} \ar@{->>}[r]&
(B\ox_\psi A)^{\ox 2}
}\big).
\end{eqnarray}
By \eqref{eq:inverse}, $\psi\,
\psiinv\,(\eta\ox\eta)=\psi\,(\eta\ox\eta)$. Hence by
\eqref{eq:quotient}, \eqref{eq:delta1} is equal to  
$$
\xymatrix{
k\ar[r]^-{\eta\ox \eta}&B\ox A \ar[r]^-{\Delta_{B\ox A}}& 
(B\ox A)^{\ox 2} \ar@{->>}[r]&
(B\ox_\psi A)^{\ox 2}
}.
$$
On the other hand, by \eqref{eq:psibar}, \eqref{eq:delta1} is equal also
to 
$$
\xymatrix{
k\ar[r]^-{\eta\ox \eta}&A\ox B\ar[r]^-{\Delta_{A\ox B}}& 
(A\ox B)^{\ox 2} \ar[r]^-{\psi^{\ox 2}}&(B\ox A)^{\ox 2} \ar@{->>}[r]&
(B\ox_\psi A)^{\ox 2}
}.
$$
With these identities at hand, the compatibility conditions between the
comultiplication and the unit follow by commutativity of diagrams \eqref{(3)}
and \eqref{segundo} on page \pageref{(3)}. The regions marked by
$\eqref{eq:Delta(1)}$ commute by the weak bialgebra axiom \eqref{eq:Delta(1)},
holding both in $A$ and $B$. 

Making use of \cite[Lemma 1.2]{BCJ}, instead of the weak bialgebra axioms 
\eqref{eq:epsilon_mu}, we will prove that their equivalent forms in
\eqref{eq:alternative_ax} hold in $B\ox_\psi A$. 
In the case of the second equality in \eqref{eq:alternative_ax}, this
means
\begin{eqnarray}\label{eq:2.a}
&&(\epsilon_{B\ox_\psi A}\ox (B\ox_\psi A))\,
(\mu_{B\ox_\psi A}\ox (B\ox_\psi A))\,
((B\ox_\psi A)\ox \Delta_{B\ox_\psi A})=\\
&&(\epsilon_{B\ox_\psi A}\ox (B\ox_\psi A))\,
(\mu_{B\ox_\psi A}\ox \mu_{B\ox_\psi A})\,
((B\ox_\psi A)\ox \Delta_{B\ox_\psi A}\, \eta_{B\ox_\psi A} \ox (B\ox_\psi A)). 
\nonumber
\end{eqnarray}
The commutative diagrams \eqref{eq:diagr_2} and \eqref{eq:diagr_3} on page
\pageref{eq:diagr_2} give rise to the commutative diagram \eqref{eq:diagr_4}
on page \pageref{eq:diagr_4} (the regions marked by \eqref{eq:alternative_ax}
commute since \eqref{eq:alternative_ax} holds both in $A$ and
$B$). Consequently, also diagram \eqref{eq:rhs} on page \pageref{eq:rhs} 
commutes -- whose right-then-down path is equal to the right hand side of
\eqref{eq:2.a}. Finally, also diagram \eqref{eq:lhs} on page \pageref{eq:lhs}
commutes -- whose right-then-down path is equal to the left hand side of 
\eqref{eq:2.a}. Since the down-then-right paths in both diagrams
\eqref{eq:rhs} and \eqref{eq:lhs} on pages  \pageref{eq:rhs} and
\pageref{eq:lhs} are equal, we have \eqref{eq:2.a} proven. 

It is proven symmetrically that also the first equality in
\eqref{eq:alternative_ax} holds in $B\ox_\psi A$. 
\end{proof}

Note that the weak inverse $\psiinv$ of a weakly invertible weakly
  comonoidal weak distributive law $\psi:A\ox B \to B\ox A$ is not
  unique. However, the double crossed product weak bialgebra $B\ox_\psi A$ in
  Theorem  \ref{thm:product_wba} does not depend on the choice of $\psiinv$
  only on its existence. 

\section{The antipode}\label{sec:antip}

For any $k$-module $A$ which carries both an algebra structure $(\mu,\eta)$
and a coalgebra structure $(\Delta,\epsilon)$, the $k$-module of
$k$-linear maps $A\to A$ carries an algebra structure via the convolution
product 
$
\varphi\ast \varphi':=\mu(\varphi\ox \varphi')\Delta
$
and the unit $\eta \epsilon$. Recall from \cite[Lemma 2.5]{BNSz:WHAI} that for
a weak bialgebra $A$, 
\begin{equation}\label{eq:conv_id}
\sqcap^R\ast \sqcap^R=\sqcap^R,\quad
\sqcap^L\ast \sqcap^L=\sqcap^L,\quad
A\ast \sqcap^R=A=\sqcap^L\ast A.
\end{equation}
With this notation, the weak Hopf algebra axioms in \eqref{eq:antip_ax} can be
written as  
$$
S \ast A = \sqcap^R,\quad
A \ast S = \sqcap^L,\quad
S\ast A \ast S=S.
$$

The aim of this section is to prove the following.

\begin{theorem}\label{thm:antipode}
If both weak bialgebras $A$ and $B$ in Theorem \ref{thm:product_wba} are weak
{\em Hopf} algebras then so is the weak wreath product $B\ox_\psi A$, for any
weakly invertible weakly comonoidal weak distributive law $\psi:A\ox B \to B
\ox A$.  
\end{theorem}

The proof starts with this.

\begin{lemma}\label{lem:antipode}
For a weak bialgebra $A$, the following assertions are equivalent.
\begin{itemize}
\item[{(1)}] $A$ is a weak Hopf algebra.
\item[{(2)}] There is a (non-unique) linear map $Z:A\to A$ such that 
$$
A\ast Z =\sqcap^L\qquad \textrm{and}\qquad
Z\ast A=\sqcap^R.
$$
\end{itemize}
\end{lemma}

\begin{proof}[Proof of Lemma \ref{lem:antipode}]
If (1) holds, then also (2) holds true with choosing $Z$ to be the antipode. 
Conversely, assume that (2) holds and put $S:=Z\ast A\ast Z$.
Note that $S$ can be written in the equivalent forms $S=\sqcap^R\ast Z$ or
$S=Z\ast \sqcap^L$. 
By the identities \eqref{eq:conv_id} we obtain
$$
A\ast S=
A\ast \sqcap^R \ast Z=
A\ast Z=
\sqcap^L
\quad \textrm{and}\quad
S\ast A=
Z\ast \sqcap^L\ast A=
Z\ast A=
\sqcap^R.
$$
Finally, 
$$
S\ast A\ast S=
Z\ast \sqcap^L\ast  \sqcap^L= 
Z\ast \sqcap^L= S. 
$$ 
This proves that $S$ is the antipode, as stated.
\end{proof}

\begin{proof}[Proof of Theorem \ref{thm:antipode}.]
Denote both antipodes in $A$ and $B$ by $S$.
We show that the map 
$$
Z:=\big(
\xymatrix{
B\ox_ \psi A\ \ar@{>->}[r]&
B\ox A \ar[r]^-{\flip}&
A\ox B \ar[r]^-{S\ox S}&
A\ox B\ar[r]^-{\psi}&
B\ox A \ar@{->>}[r]&
B\ox_ \psi A
}\big)
$$
satisfies the properties in part (2) of  Lemma \ref{lem:antipode}. 

Since the first identity in \eqref{eq:wba_id} holds in $B$ and $\sqcap^L=B\ast
S$, 
\begin{eqnarray}\label{eq:antip1}
\\
&&{\white .}\hspace{-.7cm}
(A\ox \mu\ox A^{\ox 2}\ox B)\,
(\flip_{B\ox A,A\ox B}\ox A\ox B)\,
(B\ox A\ox \Delta_{A\ox B})\,
(B\ox A^{\ox 2}\ox \eta)=\nonumber\\
&&{\white .}\hspace{-.7cm}
(A\ox B \ox A^{\ox 2}\ox \mu)\, 
(A\ox B\ox \flip_{A\ox B^{\ox 2},A})\,
(\Delta_{A\ox B}\ox B\ox A)\,
(A\ox B \ox S\ox A)\,
(A\ox \Delta\ox A)\,
\flip_{B\ox A,A}.
\nonumber
\end{eqnarray}
Also, since the second identity in \eqref{eq:wba_id} holds in $A$, 
\begin{eqnarray}\label{eq:antip2}
&&(B\ox \mu\ox B)\,
(B\ox A\ox \flip_{B,A})\,
(B\ox A\ox B\ox \sqcap^L)=\\
&&(\epsilon\ox B\ox A\ox B)\,
(\mu\ox B\ox A\ox B)\,
(A\ox \flip_{B\ox A\ox B,A})\,
(\flip_{B,A}\ox A\ox B\ox A)\,
(B\ox \Delta\ox B\ox A).
\nonumber
\end{eqnarray}
With these identities at hand, diagram \eqref{eq:antipright} on page
\pageref{eq:antipright} is seen to commute. Since also diagram
\eqref{eq:antipleft} on page \pageref{eq:antipleft} commutes and the
down-then-right paths in both diagrams coincide, we conclude that  also
their right-then-down paths are equal; that is the first condition in
part (2) of Lemma \ref{lem:antipode} holds in $B\ox_\psi A$. The second
condition follows symmetrically. 
\end{proof}

\section{Examples}\label{sec:ex}

\subsection{Wreath product of weak bialgebras} 

Since in particular distributive laws themselves are examples of weak
distributive laws, our theory includes wreath products of weak bialgebras --
induced by invertible distributive laws which are coalgebra homomorphisms. 

\begin{example} {\bf Tensor product of weak bialgebras.} For any algebras $A$
  and $B$, the twist map $A\ox B \to B\ox A$, $a\ox b\mapsto b\ox a$ is an
  invertible distributive law. If $A$ and $B$ are weak bialgebras, then it is
  also a coalgebra homomorphism. Hence by Theorem \ref{thm:product_wba}, $B\ox
  A$ is a weak bialgebra with the tensor product algebra and coalgebra
  structures. By Theorem \ref{thm:antipode}, $B\ox A$ is a weak Hopf algebra
  whenever $A$ and $B$ are so. 
\end{example}

\begin{example} {\bf The strictification of weakly equivariant Hopf algebras.}
Consider a group $G$ of finite order $n$. Then for any commutative ring $k$, the
free $k$-module $kG$ is known to be a Hopf algebra, with multiplication
obtained by the linear extension of the group multiplication and letting the
comultiplication act diagonally on the group elements. Recall from \cite{BCM}
that a $k$-algebra $A$ is said to be {\em measured} by $kG$ if there exist
algebra homomorphisms $\varphi_g:A\to A$ for all $g\in G$. Moreover, a {\em
  twisted 2-cocycle} for this measuring is a family of invertible 
elements $c_{g,h}\in A$, for all $g,h\in G$, such that $\varphi_1=\mathsf{Id}$,
$c_{1,g} =1=c_{g,1}$,
$$
\varphi_g\varphi_h=\mathsf{Ad}_{c_{g,h}}\varphi_{gh}\qquad \textrm{and}\qquad
\varphi_g(c_{h,k})c_{g,hk}=c_{g,h}c_{gh,k},\qquad \textrm{for all}\ g,h,k\in G,
$$
(where $\mathsf{Ad}_{c_{g,h}}$ denotes the inner automorphism induced
by $c_{g,h}$). It is easy to see that $\varphi_g$ is then an isomorphism
for all $g\in G$. To these data a weak Hopf algebra was associated in
\cite{MNSch}, where the authors called it as in the title. Our aim is to
describe it as a wreath product.

Denote by $M_n$ the algebra of $n\times n$ matrices with entries in $k$, and
denote the matrix units by $\{e_{g,h}\ \vert\  g,h\in G\}$. Then
$$
\psi:M_n\ox A\to A\ox M_n\qquad
e_{g,h}\ox a\mapsto 
c^{-1}_{g^{-1}h,h^{-1}}\varphi_{g^{-1}h}(a)c_{g^{-1}h,h^{-1}} \ox e_{g,h}
$$
is an invertible distributive law.

Recall that $M_n$ is a weak Hopf algebra via the comultiplication acting
diagonally on the matrix units. 
Assume that $A$ is a Hopf algebra and that the measuring and the twisted
cocycle are compatible with its coalgebra structure in the sense of
\cite[Definition 2.3]{MNSch}. That is, assume that $\varphi_g$ is a
coalgebra homomorphism and $c_{g,h}$ is a grouplike element for all $g,h\in
G$. It is straightforward to see that in this case also $\psi$ is a coalgebra
map if both in the domain and the codomain of $\psi$ the tensor product
coalgebra structure is taken. 

The corresponding double crossed product weak Hopf algebra is
isomorphic to the weak Hopf algebra defined on the $k$-module
$\widehat{kG} \ox A \ox kG$ in \cite{MNSch}, via   
$$
A\ox M_n \to \widehat{kG} \ox A \ox kG,\qquad 
a\ox e_{g,h}\mapsto \widehat{g} \ox a c^{-1}_{g^{-1}h,h^{-1}}\ox g^{-1}h, 
$$
where $\widehat{kG}$ is the $k$-linear dual of $kG$,  and
$\{\widehat{g}\, \vert\, g\in G \}$ is its basis dual to the basis $\{g\in
G\}$ of $kG$.   
\end{example}

\begin{example} {\bf The algebraic quantum torus.}
Assume $k$ to be a field, and let $N$ be a positive integer which is
not a multiple of the characteristic of $k$.  The algebra  $\langle U,V\vert
U^N=1,VU=qUV\rangle$, with $U, V$ invertible and $q \in k$ such that $q^N =
1$, is a weak Hopf algebra -- known as in the title -- via the
comultiplication   
$$ \Delta(U^nV^m)=\frac 1 N \sum_{k=1}^N(U^{k+n}V^m\ox U^{-k}V^m).$$ 
The distinguished subalgebra $\langle U\rangle$ is isomorphic to the
group algebra of the cyclic group of order $N$. It is a weak Hopf algebra via
$\Delta(U^n)=\frac 1 N \sum_{k=1}^N (U^{k+n}\ox U^{-k})$. The
subalgebra $\langle V,V^{-1} \rangle$ is isomorphic to the group algebra of
the additive group of integers. Hence it is a Hopf algebra via
$\Delta(V^m)=V^m\ox V^m$.  

The algebraic quantum torus is a double crossed product of the Hopf
algebra $\langle V,V^{-1} \rangle$ and the weak Hopf algebra $\langle
U\rangle$ with respect to the comonoidal invertible distributive law 
$$ 
\psi:\langle V,V^{-1} \rangle\ox\langle U \rangle \to 
\langle U\rangle\ox\langle V,V^{-1} \rangle,\qquad
V^m\ox U^n\mapsto q^{nm} U^n\ox V^m.
$$
\end{example}

\subsection{Wreath product of bialgebroids over a common separable Frobenius
  base algebra}  

Recall that an algebra $R$ over a commutative ring $k$ is said to be {\em
Frobenius algebra} if there exist a $k$-module map $\pi:R\to k$ and an
element $\sum_i e_i\ox f_i\in  R\ox R$ such that $\sum_i
\pi(re_i)f_i=r=\sum_i e_i\pi(f_ir)$, for all $r\in R$. It follows that for any
element $r$ of a Frobenius algebra $R$, $\sum_i re_i\ox f_i=\sum_i e_i\ox f_i
r$. We say that $R$ is a {\em separable Frobenius algebra} if in addition
$\sum_i e_i\ox f_i$ is a separability element i.e. $\sum_i e_if_i=1$. In this
case the canonical epimorphism $M\ox N \twoheadrightarrow M\ox_R N$ is split
by  
 $$
M\ox_R N \rightarrowtail M\ox N,\qquad
m\ox_R n \mapsto \sum_i m.e_i\ox f_i.n,
$$
for any right $R$-module $M$ and any left $R$-module $N$.
What is more, by \cite[Section 2.4]{BGT} any distributive law in the monoidal
category of bimodules over a separable Frobenius $k$-algebra $R$ determines a
weak distributive law in the monoidal category of $k$-modules, such that
the wreath product induced by the $R$-distributive law is isomorphic to
the weak wreath product induced by the corresponding weak $k$-distributive
law.   

As a generalization of bialgebras from commutative to non-commutative base
rings, bialgebroids were introduced by Takeuchi in \cite{Tak}. Conceptually,
an $R$-bialgebroid $A$ is an $R\ox R^{op}$-ring (i.e. an algebra in the
monoidal category of $R\ox R^{op}$-bimodules), such that the induced
monad $(-) \ox_{R\ox R^{op}} A$ on the category of $R$-bimodules (regarded as the
category of right $R\ox R^{op}$-modules) is a comonoidal monad,
see \cite{Szl:monEM}. 

It was proved by Szlach\'anyi in \cite{Szl:Fields} that a bialgebroid over a
given separable Frobenius $k$-algebra $R$ is precisely the same as a weak
bialgebra over $k$, such that the image of the right projection $\sqcap^R$ is
isomorphic to $R$.  

Let us take weak bialgebras $A$ and $B$ in which  the images
of the right projections are isomorphic as separable Frobenius algebras. 
Let us denote this common base algebra by $R$, and write $R^e := R \ox R^{op}$
for its enveloping algebra. Then $A$ and $B$ are both algebras in the monoidal
category of $R^e$-bimodules via the algebra homomorphisms 
$$
R^e \to A, \quad r\ox l \mapsto r\overline \sqcap^L(l)
\quad \textrm{and} \quad 
R^e \to B , \quad r\ox l \mapsto r\overline \sqcap^L(l)
$$ 
and we may consider a distributive law $\Psi: A\ox_{R^e} B \to
B\ox_{R^e} A$ in the category of $R^e$-bimodules. Since $R$ is a separable
Frobenius algebra, so is $R^e$. So let $\psi$ be the 
corresponding weak distributive law  
$$
\xymatrix{
A\ox B \ar@{->>}[r]&
A \ox_{R^e} B\ar[r]^-\Psi&
B \ox_{R^e} A\ \ \ar@{>->}[r]&
B\ox A}
$$
in \cite[Section 2.4]{BGT}.

\begin{proposition}
In the above setting, if $\Psi$ is an isomorphism (of $
R^e$-bimodules), then $\psi$ is weakly invertible. Moreover, if 
$$
(-)\tensor {R^e}\Psi\ :\ 
(-)\tensor {R^e} \big(A \tensor {R^e} B\big) \to 
(-)\tensor {R^e} \big(B \tensor {R^e} A)
$$ 
is a comonoidal natural transformation then $\psi$ is 
weakly comonoidal.
\end{proposition}

\begin{proof}
In terms of the inverse of $\Psi$, the weak inverse of $\psi$ is given by 
$$
\xymatrix{
B\ox A \ar@{->>}[r]&
B \ox_{R^e}A \ar[r]^-{\Psi^{-1}}&
 A \ox_{R^e}B\ \ \ar@{>->}[r]&
A\ox B.}
$$
Recall (e.g. from \cite{McCr}) that comonoidality of $\Psi$ means
commutativity of the following diagrams, for any $R$-bimodules $X$ and $Y$,
$$
\xymatrix{
(X \stac R Y) \stac {R^e} A \stac {R^e} B 
\ar[r]^-{(X\stac R Y) \stac {R^e} \Psi}
\ar[d]_-{\alpha^2_{X,Y}\ox_{R^e} B}&
(X \stac R Y) \stac {R^e} B \stac {R^e} A
\ar[d]^-{\beta^2_{X,Y}\ox_{R^e} A}\\
((X \stac {R^e} A)\stac R (Y\stac {R^e} A))\stac {R^e} B
\ar[d]_-{\beta^2}&
((X \stac {R^e} B)\stac R (Y\stac {R^e} B))\stac {R^e} A
\ar[d]^-{\alpha^2}\\
(X \stac  {R^e} A \stac {R^e} B) \stac R
(Y \stac  {R^e} A \stac {R^e} B)
\ar[r]_-{\raisebox{-8pt}{${}_{(X\stac {R^e} \Psi)\stac R (Y \stac {R^e} \Psi)}$}}& 
(X \stac  {R^e} B \stac {R^e} A) \stac R
(Y \stac  {R^e} B \stac {R^e} A)&\\
R\stac {R^e} A \stac {R^e} B 
\ar[r]^-{R\stac {R^e} \Psi}
\ar[d]_-{\alpha^0 \stac {R^e} B}&
R\stac {R^e} B \stac {R^e} A
\ar[d]^-{\beta^0 \stac {R^e} A}\\
R \stac {R^e} B 
\ar[d]_-{\beta^0}&
R \stac {R^e} A 
\ar[d]^-{\alpha^0}\\
R \ar@{=}[r]&
R}
$$
where the comonoidal structure of $(-)\ox_ {R^e} A$ is given in terms
of the comultiplication $\Delta(a)=a^1\ox a^2$ by  
$$
\alpha^2((x\tensor R y)\tensor {R^e} a)=
(x\tensor {R^e} a^1)\tensor R (y\tensor {R^e} a^2),
\quad\textrm{and}\quad
\alpha^0(r\tensor {R^e} a)=\sqcap^R(ra),
$$
and similarly for $B$. Using the index notation $\Psi(a\ox_ {R^e} b)=a_\Psi
\ox_ {R^e} b_\Psi$, and applying the monomorphisms form the $R$-module tensor
products to the $k$-module tensor products, these conditions read as 
$$
{b_\Psi}^1e_k\overline \sqcap^L(f_l)\ox 
f_k \overline \sqcap^L(e_l){a_\Psi}^1\ox 
{b_\Psi}^2e_i\ox
f_i {a_\Psi}^2=
{b^1}_\Psi e_k\overline \sqcap^L(f_l)\ox 
f_k \overline \sqcap^L(e_l) {a^1}_\Psi\ox 
{b^2}_{\Psi'} e_i\overline \sqcap^L(f_j)\ox 
f_i \overline \sqcap^L(e_j) {a^2}_{\Psi'}
$$
and
$$
\epsilon_A(rae_j)\epsilon_B(f_jbe_i)f_i=
\epsilon_B(rb_\Psi e_j)\epsilon_A(f_ja_\Psi e_i)f_i
$$
for $r\in R$, $a\in A$, $b\in B$, where we omitted the summation signs for
brevity. With these identities at hand, 
$$
\psi(a\ox b)=b_\Psi e_k\overline \sqcap^L(f_l)\ox 
f_k \overline \sqcap^L(e_l) a_\Psi
$$
and its weak inverse are easily seen to obey the weak comonoidality conditions
in Definition \ref{def:weak_comon}. 
\end{proof}

The double crossed product induced by an invertible comonoidal distributive
law $A\ox_{R^e} B \to B\ox_{R^e} A$ lives on $B\ox_{ R^e} A\cong B
e_k\overline \sqcap^L(f_l)\ox f_k \overline \sqcap^L(e_l) A$. 

\begin{example} {\bf The Drinfel'd double.} The Drinfel'd double of a finite
dimensional weak Hopf algebra $H$ is studied in several papers
\cite{Bohm:DH,CWY,Nenciu}. Regarding weak bialgebras as bialgebroids (over
separable Frobenius base algebras), the double construction of weak Hopf
algebras fits Schauenburg's double construction of finite $\times_R$-Hopf
algebras in \cite{Scha:Dual}. That is, it is a double crossed product induced
by an invertible comonoidal distributive law in the bimodule category of
$R^e$, for the base algebra $R:=\sqcap^R(H)$. The corresponding weakly
invertible weakly comonoidal weak distributive law has the explicit form 
$$
H\ox \hat H\to \hat  H\ox H,\qquad
h\ox \alpha \mapsto \alpha_2\ox h_2\langle S(h_1)\vert \alpha_1\rangle
\langle h_3 \vert \alpha_3\rangle
$$
with a weak inverse
$$
\hat  H\ox H\to H\ox \hat H,\qquad
\alpha\ox h \mapsto h_2\ox \alpha_2\langle h_1\vert \alpha_1\rangle
\langle S(h_3) \vert \alpha_3\rangle,
$$
where $\hat H$ stands for the linear dual of $H$ (a weak bialgebra via the
transposed structure)  and $\langle -\vert-\rangle:H\ox \hat H \to k$ is
the evaluation map.
\end{example}

\begin{example} {\bf Matched pairs of groupoids.}\label{ex:matched_pair}
Generalizing matched pairs of groups, matched pairs of groupoids were
introduced and studied by Andruskiewitsch and Natale in \cite{AndNat}. By
their definition, a matched pair of groupoids is a pair of groupoids $\mathcal
H$ and $\mathcal V$ over a common finite base set $\mathcal P$ 
$$
\xymatrix @C=40pt {
{\mathcal V}\ar@<3pt>[r]^-t\ar@<-3pt>[r]_b&
{\mathcal P}&
{\mathcal H},\ar@<3pt>[l]^l\ar@<-3pt>[l]_r}
$$
equipped with maps $\triangleright : {\mathcal H} {}_r\times_t {\mathcal V}
\to {\mathcal V}$ and $\triangleleft : {\mathcal H} {}_r\times_t {\mathcal V}
\to {\mathcal H}$ subject to a number of conditions, see \cite{AndNat}. These
conditions allow for the following interpretation.  

For any field $k$, consider the vector spaces $k{\mathcal H}$ and $k{\mathcal
  V}$ spanned by ${\mathcal H}$ and ${\mathcal V}$, respectively. They can be
equipped with algebra structures by requiring in terms of Kronecker's delta
symbol $h'h=\delta_{r(h'),l(h)}h'\circ h$ and $v'v=\delta_{b(v'),t(v)}v'\circ
v$. Then $k{\mathcal P}$ is a (commutative and separable Frobenius) subalgebra
in both. 
The linear map induced by    
$$
{\mathcal H} {}_r\times_t {\mathcal V}\to 
{\mathcal V} {}_b\times_l {\mathcal H},\qquad
(h,v)\mapsto (h\triangleright v,h\triangleleft v)
$$
is a distributive law $k{\mathcal H} \otimes_{k{\mathcal P}} k{\mathcal V}\to
k{\mathcal V} \otimes_{k{\mathcal P}} k{\mathcal H}$ in the category of
bimodules over $k{\mathcal P}$. It is comonoidal with respect to the
comultiplications acting diagonally on $h\in \mathcal H$ and
$v\in \mathcal V$ (with respect to which $k{\mathcal H}$ and $k{\mathcal
V}$ are weak Hopf algebras). 
It is invertible by implication (1)$\Rightarrow$(3) in \cite[Proposition
2.9]{AndNat}, thus  it gives rise to a weakly invertible weakly
comonoidal weak distributive law  
$$
k{\mathcal H} \otimes k{\mathcal V}\to 
k{\mathcal V} \otimes k{\mathcal H},\qquad
h\otimes v\mapsto \delta_{r(h),t(v)}(h\triangleright v \ox h\triangleleft v).
$$
The induced double crossed product weak Hopf algebra was analyzed
in \cite{AndNat}. 
\end{example}

\subsection{Weak wreath product of categories}
Let ${\mathcal C}$ be a category of finite object set $\mathcal P$. Then as in Example \ref{ex:matched_pair}, the vector space $k{\mathcal C}$ spanned by the morphisms of $\mathcal C$ carries a natural weak bialgebra structure via the multiplication induced by $fg=\delta_{s(f),t(g)}f\circ g$ and comultiplication induced by $g\mapsto g\ox g$, for morphisms $f$ and $g$ in $\mathcal C$, where $s$ and $t$ denote the source and the target maps, respectively.

Let $\mathcal C$ and $\mathcal D$ be categories with a common finite object set $\mathcal P$ and consider a weak distributive law $\mathcal C \mathcal D \to \mathcal D \mathcal C$ in the bicategory of spans as in \cite{Bohm:fact}. Extending it linearly in $k$, we obtain a weak distributive law $k\mathcal C \ox_{k\mathcal P} k\mathcal D \to k\mathcal D \ox_{k\mathcal P} k\mathcal C$ in the category of $k\mathcal P$ bimodules. Now since the algebra  $k\mathcal P$ possesses a separable Frobenius structure, the forgetful functor from the category of its bimodules to the category of $k$-modules possesses a so-called separable Frobenius monoidal structure \cite{Szl:adj}. Such functors were proved to preserve weak distributive laws in \cite{StrMcC}. Therefore it yields a weak distributive law in the category of $k$-modules.

\begin{example} {\bf The blown-up nothing.}
For any positive integer $n$, the algebra of $n\times n$ matrices of entries in a commutative ring $k$, with the comultiplication acting diagonally $e_{ij}\mapsto e_{ij} \ox e_{ij}$ on the matrix units, is a weak Hopf algebra. 
Since its category of modules is trivial -- in the sense that it is
equivalent to the category of $k$-modules -- it was given
in \cite{BSz:LMP} the name in the title. Below we claim that it is a double
crossed product of the sub-weak bialgebras of upper/lower triangle matrices
$U$ and $L$, respectively.  

Both algebras $L$ and $U$ are spanned by morphisms of a category of $n$ objects and precisely one morphism $i\to j$ whenever $i\geq j$ and $i\leq j$, respectively. There is a weak distributive law in the bicategory of spans
$$
(i\geq   j \leq k)\mapsto (i\leq n \geq k).
$$
The corresponding weak distributive law 
$$
\psi: L\ox U\to U\ox L,\qquad e_{ij}\ox e_{lk}\mapsto \delta_{j,l}e_{in}\ox
e_{nk},\qquad \textrm{for}\ i\geq j, l\leq k
$$
in the category of $k$-modules possesses a weak inverse
$$
\psiinv: U\ox L\to L\ox U,\qquad e_{lk} \ox e_{ij}\mapsto \delta_{k,i}e_{l1}\ox
e_{1j},\qquad \textrm{for}\ i\geq j, l\leq k.
$$
They are weakly comonoidal and the corresponding double crossed product weak
bialgebra is that of $n\times n$ matrices, with the matrix units
  $\{e_{in}\ox e_{nk}\vert 1\leq i,k\leq n\}$. 
\end{example}
\vfill 
\eject

\pagestyle{empty}
\appendix
\begin{landscape} 

\section*{appendix: Diagrams}
{\small
\begin{equation}\label{(3)}
\scalebox{.9}{
{\color{white} .}\hspace{-1cm}
\xymatrix @R=10pt @C=5pt{
k \ar[rrr]^-{\eta\ox \eta\ox \eta\ox \eta}\ar[rd]^-{\eta\ox \eta}
\ar[dddddd]_-{\eta\ox \eta}\ar@{}[rrdddddd]|-{\red \eqref{eq:Delta(1)}}
\ar@{}[rrrrdd]|-{\red \eqref{eq:Delta(1)}}&&&
A\ox B \ox B\ox A
\ar[r]^-{\raisebox{6pt}{${}_{\Delta_{A\ox B}\ox \Delta_{B\ox A}}$}}&
(A\ox B)^{\ox 2}\ox (B\ox A)^{\ox 2}
\ar[dd]_-{A\ox B\ox A\ox \mu\ox A\ox B\ox A}
\ar[r]^-{\raisebox{6pt}{${}_{\psi\ox A\ox B\ox (B\ox A)^{\ox 2}}$}}
\ar[rdd]^-{\raisebox{7pt}{${}_{\psi\ox A\ox B\ox (B\ox A)^{\ox 2}}$}}&
B\ox A^{\ox 2}\ox B\ox (B\ox A)^{\ox 2}
\ar[r]^-{\raisebox{6pt}{${}_{B\ox A\ox \psi\ox (B\ox A)^{\ox 2}}$}}&
(B\ox A)^{\ox 4}\ar[dd]_-{\psi\psiinv\ox (B\ox A)^{\ox 3}}\ar@{->>}[r]
\ar[rdd]^-{(\psi\psiinv)^{\ox 4}}&
(B\stac \psi A)^{\ox 4}\ar@{>->}[dd]\\
&B\ox A\ar[rd]^-{\Delta_{B\ox A}}&&&&&&\\
&&(B\ox A)^{\ox 2}\ar[d]_-{B\ox \eta\ox A\ox B\ox A}
\ar[r]^-{\raisebox{6pt}{${}_{\eta\ox (B\ox A)^{\ox 2}}$}} 
\ar@{}[rd]|-{\red \eqref{eq:inverse}}&
(A\ox B)^{\ox 2}\ox A\ar[r]^-{\Delta_{A\ox B}\ox A\ox B\ox A}
\ar[d]^-{\psi\ox A\ox B\ox A} \ar@{}[rdd]|-{\red \eqref{eq:psibar}}&
(A\ox B)^{\ox 3}\ox A\ar[ddd]^-{\psi^{\ox 2}\ox A\ox B\ox A}&
B\ox A^{\ox 2}\ox B^{\ox 2}\ox A\ox B\ox A
\ar[r]^-{B\ox A \ox \psi\ox (B\ox A)^{\ox 2}}
\ar[d]^-{B\ox A^{\ox 2}\ox \mu\ox A\ox B\ox A}&
(B\ox A)^{\ox 4}\ar[d]^-{B\ox A\ox B\ox \psi\ox A\ox B\ox A}
\ar@{}[rddd]|(.7){\red \eqref{eq:onleft}\eqref{eq:onright}}&
(B\ox A)^{\ox 4}\ar[d]^-{B\ox A \ox B \ox \psi\ox A\ox B\ox A}\\
&& B\ox A^{\ox 2}\ox B\ox A\ar[dd]_-{\Delta_{B\ox A}\ox A\ox B\ox A}
\ar[r]^-{\psi\psiinv\ox A\ox B\ox A}
\ar@{}[rdd]|-{\red\eqref{eq:quotient}}&
B\ox A^{\ox 2}\ox B\ox A \ar[d]^-{\Delta_{B\ox A}\ox A\ox B\ox A}&&
B\ox A^{\ox 2}\ox (B\ox A)^{\ox 2}
\ar[ldd]^-{\raisebox{-7pt}{${}_{B\ox A\ox \psi\ox A\ox B\ox A}$}}
\ar@{}[rd]|-{\red \eqref{eq:wdl}}&
B\ox A \ox B^{\ox 2}\ox A^{\ox 2}\ox B\ox A
\ar[d]^-{B\ox A\ox \mu\ox A^{\ox  2}\ox B\ox A}&
B\ox A \ox B^{\ox 2}\ox A^{\ox 2}\ox B\ox A
\ar[dd]^-{B\ox A\ox \mu\ox \mu\ox B\ox A}\\
&&&(B\ox A)^{\ox 2}\ar[d]^-{(\psi\psiinv)^{\ox 2}\ox A\ox B\ox A}&&&
(B\ox A)^{\ox 2} \ox A\ox B\ox A\ar[d]^-{B\ox A\ox B \ox \mu\ox B\ox A}&\\
&&(B\ox A)^{\ox 2} \ox A\ox B\ox A
\ar[r]^-{\raisebox{6pt}{${}_{(\psi\psiinv)^{\ox 2}\ox A\ox B\ox A}$}}
\ar[d]_-{B\ox A\ox B\ox \mu\ox B\ox A}\ar@{}[rrrrd]|-{\red\eqref{eq:bilin}}&
(B\ox A)^{\ox 2} \ox A\ox B\ox A \ar@{=}[r]&
(B\ox A)^{\ox 2} \ox A\ox B\ox A \ar@{=}[urr]&&
(B\ox A)^{\ox 3}\ar[d]^(.4){(B\ox A)^{\ox 2}\ox \psi\psiinv}&
(B\ox A)^{\ox 3}\ar@{->>}[d]
\ar[ld]_(.3){(\psi\psiinv)^{\ox 3}}\\
B\ox A \ar[rr]_{\!\!\!(\Delta_{B\ox A}\ox B\ox A)\Delta_{B\ox A}}&&
(B\ox A)^{\ox 3}\ar[rrrr]_-{(\psi\psiinv)^{\ox 3}}&&&&
(B\ox A)^{\ox 3}&
\ \ (B\stac \psi A)^{\ox 3}\ar@{>->}[l]
}}
\end{equation}
\bigskip

\begin{equation}\label{segundo}
\scalebox{.9}{
{\color{white} .}\hspace{-1cm}
\xymatrix @R=10pt @C=14pt{
k \ar[rrr]^-{\eta\ox \eta\ox \eta\ox \eta}\ar[rd]^-{\eta}
\ar[dddddd]_-{\eta\ox \eta}\ar@{}[rrdddddd]|-{\red \eqref{eq:Delta(1)}}
\ar@{}[rrrrdd]|-{\red \eqref{eq:Delta(1)}}&&&
A\ox B \ox B\ox A
\ar[r]^-{\raisebox{6pt}{${}_{\Delta_{A\ox B\ox B\ox A}}$}}&
(A\ox B^{\ox 2}\ox A)^{\ox 2}
\ar[dd]_-{A\ox B^{\ox 2}\ox \mu\ox B^{\ox 2}\ox A}
\ar[r]^-{\raisebox{6pt}{${}_{\psi\ox B\ox A^{\ox 2}\ox B^{\ox 2}\ox A}$}}
\ar[rdd]^-{\raisebox{7pt}{${}_{\psi\ox B\ox A^{\ox 2}\ox B^{\ox 2}\ox A}$}}&
(B\ox A)^{\ox 2}\ox A\ox B^{\ox 2}\ox A
\ar[r]^-{\raisebox{6pt}{${}_{(B\ox A)^{\ox 2}\ox \psi\ox B\ox A}$}}&
(B\ox A)^{\ox 4}\ar[dd]_-{\psi\psiinv\ox (B\ox A)^{\ox 3}}\ar@{->>}[r]
\ar[rdd]^-{(\psi\psiinv)^{\ox 4}}&
(B\stac \psi A)^{\ox 4}\ar@{>->}[dd]\\
&A\ar[rd]^-{\Delta}&&&&&&\\
&&A^{\ox 2}\ar[d]_-{(\eta\ox A)^{\ox 2}}
\ar[r]^-{\raisebox{6pt}{${}_{A\ox \eta^{\ox 2}\ox A}$}} 
\ar@{}[rd]|-{\red \eqref{eq:inverse}}&
A\ox B^{\ox 2}\ox A\ar[r]^-{\Delta_{A\ox B^{\ox 2}}\ox A}
\ar[d]^-{\psi\ox B\ox A} \ar@{}[rdd]|-{\red \eqref{eq:psibar}}&
A\ox B^{\ox 2}\ox A\ox B^{\ox 2}\ox A\ar[ddd]^-{(\psi \ox B)^{\ox 2} \ox A}&
(B\ox A)^{\ox 2}\ox A\ox B^{\ox 2}\ox A
\ar[r]^-{(B\ox A)^{\ox 2} \ox \psi\ox B\ox A}
\ar[d]^-{B\ox A\ox B\ox \mu\ox B^{\ox 2}\ox A}&
(B\ox A)^{\ox 4}\ar[d]^-{B\ox A\ox B\ox \psi\ox A\ox B\ox A}
\ar@{}[rddd]|(.7){\red \eqref{eq:onleft}\eqref{eq:onright}}&
(B\ox A)^{\ox 4}\ar[d]^-{B\ox A \ox B \ox \psi\ox A\ox B\ox A}\\
&& (B\ox A)^{\ox 2}\ar[dd]_-{\Delta_{B\ox A\ox B}\ox A}
\ar[r]^-{\psi\psiinv\ox B\ox A}
\ar@{}[rdd]|(.3){\red\eqref{eq:quotient}}&
(B\ox A)^{\ox 2}\ar[d]^-{\Delta_{B\ox A\ox B}\ox A}&&
(B\ox A)^{\ox 2}\ox B^{\ox 2}\ox A
\ar[ldd]^-{\raisebox{-7pt}{${}_{B\ox A\ox B\ox \psi\ox B\ox A}$}}
\ar@{}[rd]|-{\red \eqref{eq:wdl}}&
B\ox A \ox B^{\ox 2}\ox A^{\ox 2}\ox B\ox A
\ar[d]^-{B\ox A\ox B^{\ox 2}\ox \mu\ox B\ox A}&
B\ox A \ox B^{\ox 2}\ox A^{\ox 2}\ox B\ox A
\ar[dd]^-{B\ox A\ox \mu\ox \mu\ox B\ox A}\\
&&&(B\ox A\ox B)^{\ox 2}\ox A
\ar[d]^-{(\psi\psiinv\ox B)^{\ox 2}\ox A}&&&
B\ox A\ox B\ox (B\ox A)^{\ox 2} \ar[d]^-{B\ox A\ox \mu\ox A\ox B\ox A}&\\
&&(B\ox A\ox B)^{\ox 2} \ox A
\ar[r]^-{\raisebox{6pt}{${}_{(\psi\psiinv\ox B)^{\ox 2}\ox A}$}}
\ar[d]_-{B\ox A\ox \mu\ox A\ox B\ox A}\ar@{}[rrrrd]|-{\red\eqref{eq:bilin}}&
(B\ox A\ox B)^{\ox 2} \ox A \ar@{=}[r]&
(B\ox A\ox B)^{\ox 2} \ox A \ar@{=}[urr]&&
(B\ox A)^{\ox 3}\ar[d]^-{(B\ox A)^{\ox 2}\ox \psi\psiinv}&
(B\ox A)^{\ox 3}\ar@{->>}[d]
\ar[ld]_(.3){(\psi\psiinv)^{\ox 3}}\\
B\ox A \ar[rr]_{(B\ox A\ox \Delta_{B\ox A})\Delta_{B\ox A}}&&
(B\ox A)^{\ox 3}\ar[rrrr]_-{(\psi\psiinv)^{\ox 3}}&&&&
(B\ox A)^{\ox 3}&
\ \ (B\stac \psi A)^{\ox 3}\ar@{>->}[l]
}}
\end{equation}}
\newpage

\begin{equation}\label{eq:diagr_2}
\xymatrix @C=22pt{
A\ox B^{\ox 2}\ox (A\ox B)^{\ox 2}\ar[rr]^-{A\ox \mu\ox (A\ox B)^{\ox 2}}
\ar[d]_-{A\ox B\ox {\psiinv}\ox B\ox A\ox B}
\ar@{}[rrd]|-{\red\eqref{eq:wdl}}&&
(A\ox B)^{\ox 3}\ar[rrr]^-{\psi\ox A\ox B\ox \psi}
\ar[d]_-{A\ox \psiinv\ox B\ox A\ox B}
\ar[rd]^-{A\ox {\psiinv}^{\ox 2}\ox B}
\ar@{}[rrrd]|-{\red\eqref{eq:wdl}\ {\red \eqref{eq:onright}}}&&&
B\ox A^{\ox 2}\ox B^{\ox 2}\ox A\ar[d]^-{B\ox \mu\ox \mu\ox A}\\
(A\ox B)^{\ox 2}\ox B\ox A\ox B
\ar[r]^-{\raisebox{7pt}{${}_{A\ox \psiinv\ox B^{\ox 2}\ox A\ox B}$}}
\ar[d]_-{(A\ox B)^{\ox 2}\ox \psiinv\ox B}&
A^{\ox 2}\ox B^{\ox 3}\ox A\ox B
\ar[r]^-{\raisebox{6pt}{${}_{A^{\ox 2}\ox \mu\ox B\ox A\ox B}$}}&
A^{\ox 2}\ox B^{\ox 2}\ox A\ox B
\ar[dd]_-{A^{\ox 2}\ox B\ox \psiinv\ox B}
\ar@{}[rrdd]|-{\red\eqref{eq:bilin}}&
A^{\ox 2}\ox B\ox A\ox B^{\ox 2}
\ar[r]^-{\raisebox{6pt}{${}_{\mu\ox B\ox A\ox \mu}$}}&
(A\ox B)^{\ox 2}\ar[r]^-{\psi^{\ox 2}}
\ar[dd]_-{\psiinv\, \psi\ox \ox A\ox B}
\ar@{}[rdd]|-{\red\eqref{eq:counit}}&
(B\ox A)^{\ox 2}\ar[d]^-{\epsilon\ox \epsilon\ox B\ox A}\\
(A\ox B)^{\ox 3}\ox B\ar[d]_-{A\ox \psiinv\ox B\ox A\ox B^{\ox 2}}
&&&&&B\ox A\\
A^{\ox 2}\ox B^{\ox 2}\ox A\ox B^{\ox 2}
\ar[rr]_-{A^{\ox 2}\ox \mu\ox A\ox B^{\ox 2}}&&
A^{\ox 2}\ox B\ox A\ox B^{\ox 2}\ar[rr]_-{\mu\ox B\ox A\ox \mu}&&
(A\ox B)^{\ox 2}\ar[r]_-{A\ox B\ox \psi}&
A\ox B^{\ox 2}\ox A\ar[u]_-{\epsilon\ox \epsilon\ox B\ox A}
}
\end{equation}

\begin{equation}\label{eq:diagr_3}
\xymatrix @R=12pt{
A\ox B^{\ox 2}\ox A\ox B  
\ar[rrr]^-{A\ox B\ox \Delta_{B\ox A}\ox B}
\ar[d]_-{A\ox B\ox \psiinv \ox B}
\ar@{}[rrrrrd]|-{\red\eqref{eq:psi}}&&& 
A\ox B^{\ox 2}\ox (A\ox B)^{\ox 2} 
\ar[rr]^-{A\ox B\ox \psiinv \ox B\ox A\ox B}&& 
(A\ox B)^{\ox 2}\ox B\ox A\ox B     
\ar[d]^-{(A\ox B)^{\ox 2}\ox \psiinv\ox B}\\
(A\ox B)^{\ox 2}\ox B           
\ar[rrr]^-{A\ox B \ox \Delta_{A\ox B}\ox B}
\ar[rdd]^-{\quad A\ox B\ox A\ox \eta \ox B^{\ox 2}}
\ar[dddd]_-{A\ox B\ox A\ox \mu}&&&
(A\ox B)^{\ox 3}\ox B   
\ar[rr]^-{A\ox B\ox \psiinv\psi\ox A \ox B^{\ox 2}}
\ar[dd]^-{A\ox \psiinv\ox B\ox A \ox B^{\ox 2}}
\ar@{}[rrdd]|-{\red\eqref{eq:onright}}&&
(A\ox B)^{\ox 3}\ox B      
\ar[d]^-{A\ox \psiinv\ox B\ox A \ox B^{\ox 2}}\\
&&&&&A^{\ox 2}\ox B^{\ox 2}\ox A\ox B^{\ox 2} 
\ar[d]^-{A^{\ox 2}\ox \mu\ox A\ox B^{\ox 2}}\\
&A\ox B\ox A\ox B^{\ox 3}      
\ar[d]^-{A\ox B\ox \Delta_{A\ox B}\ox B^{\ox 2}}
\ar@{}[rruu]|-{\red \eqref{eq:alternative_ax}}&&    
A^{\ox 2}\ox B^{\ox 2}\ox A\ox B^{\ox 2}
\ar[rr]^-{A^{\ox 2}\ox \mu\ox A\ox B^{\ox 2}}&&
A^{\ox 2}\ox B\ox A\ox B^{\ox 2}          
\ar[d]^-{A^{\ox 2}\ox \epsilon\ox A\ox B^{\ox 2}}\\
&(A\ox B)^{\ox 3}\ox B^{\ox 2}      
\ar[rr]^-{A\ox \psiinv\ox B\ox A\ox B^{\ox 3}}&&
A^{\ox 2}\ox B^{\ox 2}\ox A\ox B^{\ox 3}   
\ar[r]^-{\raisebox{6pt}{${}_{A^{\ox 2}\ox \mu\ox A\ox \mu\ox B}$}}&
A^{\ox 2}\ox B\ox A\ox B^{\ox 2}    
\ar[r]^-{A^{\ox 2}\ox \epsilon\ox A\ox B^{\ox 2}}& 
A^{\ox 3}\ox B^{\ox 2}  
\ar[d]^-{A^{\ox 3}\ox \mu}\\
(A\ox B)^{\ox 2} 
\ar[r]_-{\raisebox{-6pt}{${}_{A\ox B\ox A\ox \eta\ox B}$}}&
A\ox B\ox A\ox B^{\ox 2} 
\ar[r]_-{\raisebox{-6pt}{${}_{A\ox B\ox \Delta_{A\ox B}\ox B}$}}&  
(A\ox B)^{\ox 3}\ox B    
\ar[r]_-{\raisebox{-6pt}{${}_{A\ox \psiinv\ox B\ox A \ox B^{\ox 2}}$}}&
A^{\ox 2}\ox B^{\ox 2}\ox A\ox B^{\ox 2}  
\ar[r]^-{\raisebox{6pt}{${}_{A^{\ox 2}\ox \mu\ox A\ox \mu}$}}
\ar[d]^-{\mu\ox B^{\ox 2}\ox A\ox B^{\ox 2}}&
A\ox (A\ox B)^{\ox 2}   
\ar[r]^-{A^{\ox 2}\ox \epsilon\ox A\ox B}&
A^{\ox 3}\ox B    
\ar[d]^-{\mu\ox A\ox B}\\
&&&(A\ox B^{\ox 2})^{\ox 2}  
\ar[r]_-{(A\ox \mu)^{\ox 2}}&
(A\ox B)^{\ox 2}  
\ar[r]_-{A\ox \epsilon\ox A\ox B}&
A^{\ox 2}\ox B  
}
\end{equation}
\eject

\begin{equation}\label{eq:diagr_4}
\xymatrix{
A\ox B^{\ox 2}\ox A\ox B\ar[rr]^-{A\ox B^{\ox 2}\ox \eta\ox A\ox B}
\ar[d]_-{A\ox B\ox \psiinv\ox B}
\ar[rd]^-{\quad A\ox B\ox \Delta_{B\ox A}\ox B}&&
A\ox B^{\ox 2}\ox A^{\ox 2}\ox B
\ar[rr]^-{A\ox B\ox \Delta_{B\ox A}\ox A\ox B}&&
A\ox B^{\ox 2}\ox A\ox B\ox A^{\ox 2}\ox B
\ar[d]^-{A\ox \mu\ox A\ox B\ox \mu\ox B}\\
(A\ox B)^{\ox 2}\ox B\ar[d]_-{A\ox B\ox A\ox \mu}
\ar@{}[rddd]|-{\red \eqref{eq:diagr_3}}&
A\ox B^{\ox 2}\ox (A\ox B)^{\ox 2}
\ar[r]^-{\raisebox{6pt}{${}_{A\ox\mu\ox (A\ox B)^{\ox 2}}$}}
\ar[d]^-{A\ox B\ox {\psiinv}^{\ox 2}\ox B}&
(A\ox B)^{\ox 3}\ar[r]^-{\psi\ox A\ox B\ox \psi}&
B\ox A^{\ox 2}\ox B^{\ox 2}\ox A\ar[ddd]^-{B\ox \mu\ox \mu\ox A}
\ar@{}[rdd]|-{\red \eqref{eq:alternative_ax}}&
(A\ox B)^{\ox 3}\ar[d]^-{\psi\ox A\ox B\ox \psi}\\
(A\ox B)^{\ox 2}\ar[d]_-{A\ox B\ox A\ox \eta\ox B}&
(A\ox B)^{\ox 3}\ox B\ar[d]^-{A\ox \psiinv\ox B\ox A\ox B^{\ox 2}}&&&
B\ox A^{\ox 2}\ox B^{\ox 2}\ox A\ar[d]^-{B\ox \mu\ox \mu\ox A}\\
A\ox B\ox A\ox B^{\ox 2}\ar[d]_-{A\ox B\ox \Delta_{A\ox B}\ox B}&
A^{\ox 2}\ox B^{\ox 2}\ox A\ox B^{\ox 2}
\ar[d]^-{A^{\ox 2}\ox \mu\ox A\ox B^{\ox 2}}&&&
(B\ox A)^{\ox 2}\ar[d]^-{B\ox \epsilon\ox B\ox A}\\
(A\ox B)^{\ox 3}\ox B\ar[d]_-{A\ox \psiinv\ox B\ox A\ox B^{\ox 2}}&
A^{\ox 2}\ox B\ox A\ox B^{\ox 2}
\ar[r]^-{\raisebox{6pt}{${}_{A^{\ox 2}\ox \epsilon\ox A\ox B^{\ox 2}}$}}
\ar@{}[rruu]|-{\red \eqref{eq:diagr_2}}&
A^{\ox 3}\ox B^{\ox 2}\ar[d]^-{\mu\ox A\ox \mu}&
(B\ox A)^{\ox 2}\ar[r]^-{B\ox \epsilon\ox B\ox A}&
B^{\ox 2}\ox A\ar[d]^-{\epsilon\ox B\ox A}\\
A^{\ox 2}\ox B^{\ox 2}\ox A\ox B^{\ox 2}\ar[r]_-{\mu\ox \mu\ox A\ox \mu}&
(A\ox B)^{\ox 2}\ar[r]_-{A\ox \epsilon\ox A\ox B}&
A^{\ox 2}\ox B\ar[r]_-{\epsilon\ox A \ox B}&
A\ox B\ar[r]_-{\psi}&
B\ox A
}
\end{equation}
\eject

\begin{equation}\label{eq:rhs}
\scalebox{.9}{
{\color{white} .}\hspace{-.7cm}
\xymatrix @C=2pt{
(B\stac \psi A)^{\ox 2}
\ar[rr]^-{(B\ox_\psi A)\ox \eta\ox \eta\ox (B\ox_\psi A)}&&
(B\stac \psi A)\ox B\ox A\ox (B\stac \psi A)
\ar[rrr]^-{(B\ox_\psi A)\ox \Delta_{B\ox A}\ox (B\ox_\psi A)}&&&
(B\stac \psi A)\ox (B\ox A)^{\ox 2}\ox (B\stac \psi A)
\ar@{->>}[r]&
(B\stac \psi A)^{\ox 4}\ar@{>->}[d]\\
(B\ox A)^{\ox 2}
\ar[rr]^-{B\ox A\ox \eta\ox \eta\ox B\ox A}
\ar@{->>}[u]\ar[d]^-{B\ox A\ox \eta\ox B\ox A}\ar@{=}@/_4pc/[ddd]&&
(B\ox A)^{\ox 3}
\ar[rrr]^-{B\ox A\ox \Delta_{B\ox A}\ox B\ox A}&&&
(B\ox A)^{\ox 4} \ar@{->>}[u]
\ar[dl]_-{\psiinv\ox (B\ox A)^{\ox 2}\ox \psiinv\quad}
\ar[d]^-{(B\ox \psi\ox A)^{\ox 2}}
\ar[r]^-{(\psi \psiinv)^{\ox 4}}&
(B\ox A)^{\ox 4}
\ar[d]^-{(B\ox \psi\ox A)^{\ox 2}}\\
B\ox A\ox B^{\ox 2}\ox A
\ar[r]^-{\raisebox{6pt}{${}_{B\ox A\ox B\ox \psiinv}$}}
\ar[dd]^-{B\ox A\ox \mu\ox A}
\ar@{}[rdd]|(.3){\red\qquad\qquad \eqref{eq:wdl} }&
(B\ox A)^{\ox 2}\ox B
\ar[r]^-{\raisebox{6pt}{${}_{\psiinv\ox B\ox A\ox B}$}}
\ar[d]^-{B\ox A\ox \psiinv\ox B}&
A\ox B^{\ox 2}\ox A\ox B
\ar[r]^-{\raisebox{6pt}{${}_{A\ox B^{\ox 2}\ox \eta \ox A\ox B}$}}
\ar[d]^-{A\ox B\ox \psiinv \ox B}
\ar@{}[rrdddd]|-{\red \eqref{eq:diagr_4}}&
A\ox B^{\ox 2}\ox A^{\ox 2}\ox B
\ar[r]^-{\raisebox{6pt}{${}_{A\ox B\ox \Delta_{B\ox A}\ox A\ox B}$}}&
A\ox B^{\ox 2}\ox A\ox B\ox A^{\ox 2}\ox B
\ar[dddd]^-{A\ox \mu\ox A\ox B\ox \mu\ox B}
\ar@{}[rdddd]|(.35){\red \qquad \eqref{eq:wdl}\ {\red \eqref{eq:onright}}}&
(B^{\ox 2}\ox A^{\ox 2})^{\ox 2}
\ar@{}[rddd]|-{\red \eqref{eq:onleft}\eqref{eq:onright}}
\ar[dddd]^-{\mu\ox A^{\ox 2}\ox B^{\ox 2}\ox \mu}& 
(B^{\ox 2}\ox A^{\ox 2})^{\ox 2}
\ar[dd]^-{(\mu\ox \mu)^{\ox 2}}\\
&B\ox A^{\ox 2}\ox B^{\ox 2}\ar[d]^-{B\ox A^{\ox 2}\ox \mu}&
(A\ox B)^{\ox 2}\ox B\ar[d]^-{A\ox B\ox A\ox \mu}&&&&\\
(B\ox A)^{\ox 2}\ar[r]^-{B\ox A\ox \psiinv}&
B\ox A^{\ox 2}\ox B\ar[d]_-{B\ox A^{\ox 2}\ox \eta\ox B}&
(A\ox B)^{\ox 2}\ar[d]^-{A\ox B\ox A\ox \eta \ox B}&&&&
(B\ox A)^{\ox 2}\ar@{->>}[d]\\
&B\ox A^{\ox 2}\ox B^{\ox 2}\ar[d]_-{B\ox A\ox \Delta_{A\ox B}\ox B}&
(A\ox B)^{\ox 2}\ox B\ar[d]^-{A\ox B\ox \Delta_{A\ox B}\ox B}&&&&
(B\stac \psi A)^{\ox 2}\ar@{>->}[d]\\
&B\ox A\ox (A\ox B)^{\ox 2}\ox B
\ar[r]^-{\raisebox{6pt}{${}_{\psiinv \ox (A\ox B)^{\ox 2}\ox B}$}}
\ar[dd]_-{B\ox \mu\ox B\ox A \ox B^{\ox 2}}
\ar@{}[rdd]|-{\red \eqref{eq:wdl}}&
(A\ox B)^{\ox 3}\ox B\ar[d]^-{A\ox\psiinv\ox B\ox A\ox B^{\ox 2}}&&
(A\ox B)^{\ox 3}\ar[r]^-{\psi\ox A\ox B\ox \psi}&
B\ox A^{\ox 2}\ox B^{\ox 2}\ox A
\ar[r]^-{B\ox \mu\ox \mu\ox A}&
(B\ox A)^{\ox 2}\ar[dd]^-{\epsilon\ox \epsilon\ox B\ox A}\\
&&A^{\ox 2}\ox B^{\ox 2}\ox A\ox B^{\ox 2}
\ar[d]^-{\mu \ox B^{\ox 2}\ox A\ox B^{\ox 2}}&&&&\\
&(B\ox A)^{\ox 2}\ox B^{\ox 2}
\ar[r]_-{\psiinv\ox B\ox A \ox B^{\ox 2}}&
(A\ox B^{\ox 2})^{\ox 2}\ar[r]_-{(A\ox \mu)^{\ox 2}}&
(A\ox B)^{\ox 2}\ar[r]_-{A\ox \epsilon\ox A\ox B}&
A^{\ox 2}\ox B\ar[r]_-{\epsilon\ox A\ox B}&
A\ox B\ar[r]^-\psi\ar[d]_-\psi&
B\ox A\ar@{->>}[d]\ar[ld]_-{\psi\psiinv}\\
&&&&&B\ox A&
\ \ B\stac\psi A\ar@{>->}[l]
}}
\end{equation}

\eject
\begin{equation}\label{eq:lhs}
{\color{white} .}\hspace{-.6cm}
\xymatrix @C=10pt{
(B\stac \psi A)^{\ox 2}\ \ \ar@{>->}[rr]&&
(B\ox A)^{\ox 2}\ar[rr]^-{B\ox A\ox \Delta_{B\ox A}}
\ar[d]^-{B\ox A\ox \psiinv}
\ar@{}[rrdd]|-{\red\eqref{eq:psi}}&&
(B\ox A)^{\ox 3}\ar@{->>}[rr]
\ar[d]^-{(B\ox A)^{\ox 2}\ox \psi\psiinv}
\ar[rrd]^-{(\psi\psiinv)^{\ox 3}}&&
(B\stac\psi A)^{\ox 3}\ar@{>->}[d]\\
(B\ox A)^{\ox 2}\ar[dd]_-{B\ox A\ox \psiinv}
\ar@{->>}[u]\ar[rru]^-{(\psi\psiinv)^{\ox 2}}&&
B\ox A^{\ox 2}\ox B\ar[d]^-{B\ox A\ox \Delta_{A\ox B}}&&
(B\ox A)^{\ox 3}\ar@{=}[r]
\ar[d]^-{B\ox A\ox \psiinv\ox B\ox A}&
(B\ox A)^{\ox 3}\ar[d]^-{B\ox \psi\ox A\ox B\ox A}&
(B\ox A)^{\ox 3}\ar[d]^-{B\ox \psi\ox A\ox B\ox A}\\
&B\ox A\ox (A\ox B)^{\ox 2}\ar@{=}[d]
\ar[r]^-{\raisebox{8pt}{${}_{\psi\psiinv\ox (A\ox B)^{\ox 2}}$}}&
B\ox A\ox (A\ox B)^{\ox 2}\ar[rr]^-{B\ox A^{\ox 2}\ox B\ox \psi}&&
B\ox A^{\ox 2}\ox B^{\ox 2}\ox A\ar@{=}[d]
\ar@{}[rdd]|(.35){\red\qquad \eqref{eq:wdl}\ {\red \eqref{eq:onright}}}&
B^{\ox 2}\ox A^{\ox 2}\ox B\ox A\ar[dd]^-{B^{\ox 2}\ox\mu\ox B\ox A}
\ar@{}[rdd]|(.35){\red\qquad \eqref{eq:onleft}\eqref{eq:onright}}&
B^{\ox 2}\ox A^{\ox 2}\ox B\ox A\ar[d]^-{\mu\ox\mu\ox B\ox A}\\
B\ox A^{\ox 2}\ox B\ar[d]_-{B\ox A^{\ox 2}\ox \eta\ox B}
\ar[r]^-{B\ox A\ox \Delta_{A\ox B}}
\ar@{}[rdddd]|(.4){\qquad \red\eqref{eq:alternative_ax}}&
B\ox A\ox (A\ox B)^{\ox 2}
\ar[r]^-{\raisebox{8pt}{${}_{B\ox A^{\ox 2}\ox B \ox \psi}$}}
\ar[d]^-{B\ox \mu\ox B\ox A\ox B}&
B\ox A^{\ox 2}\ox B^{\ox 2}\ox A\ar[d]^-{B\ox \mu\ox B^{\ox 2}\ox A}
\ar[rr]^-{\psi\psiinv\ox A\ox B^{\ox 2}\ox A}
\ar@{}[rrd]|-{\red \eqref{eq:bilin}}&&
B\ox A^{\ox 2}\ox B^{\ox 2}\ox A\ar[d]^-{B\ox \mu\ox B^{\ox 2}\ox A}&&
(B\ox A)^{\ox 2}\ar@{->>}[d]\\
B\ox A^{\ox 2}\ox B^{\ox 2}\ar[d]_-{B\ox A\ox \Delta_{A\ox B}\ox B}&
(B\ox A)^{\ox 2}\ox B\ar[d]^-{\psiinv\ox B\ox A\ox B}&
B\ox A\ox B^{\ox 2}\ox A
\ar[r]^-{\raisebox{8pt}{${}_{\psiinv\ox B^{\ox 2}\ox A}$}}
\ar[dd]^-{\psiinv\ox B^{\ox 2}\ox A}
\ar@{}[rdd]|(.35){\red\qquad \eqref{eq:onleft}}&
A\ox B^{\ox 3}\ox A
\ar[r]^-{\raisebox{8pt}{${}_{\psi\ox B^{\ox 2}\ox A}$}}
\ar[d]^-{A\ox \mu\ox B\ox A}
\ar@{}[rrd]|-{\red\eqref{eq:wdl}}&
B\ox A\ox B^{\ox 2}\ox A
\ar[r]^-{\raisebox{8pt}{${}_{B\ox \psi\ox B\ox A}$}}&
B^{\ox 2}\ox A\ox B\ox A\ar[d]^-{\mu\ox A\ox B\ox A}&
(B\stac\psi A)^{\ox 2}\ar@{>->}[d]\\
B\ox A\ox (A\ox B)^{\ox 2}\ox B\ar[d]_-{B\ox \mu\ox B\ox A\ox B^{\ox 2}}&
A\ox B^{\ox 2}\ox A\ox B\ar[d]^-{A\ox \mu\ox A\ox B}&&
A\ox B^{\ox 2}\ox A\ar[d]^-{\psiinv\psi\ox B\ox A}
\ar[rr]^-{\psi\ox B\ox A}
\ar@{}[rrd]|-{\red \eqref{eq:counit}}&&
(B\ox A)^{\ox 2}\ar[d]^-{\epsilon\ox \epsilon\ox B\ox A}&
(B\ox A)^{\ox 2}\ar[l]_-{B\ox A\ox \psi\psiinv}
\ar[d]^-{\epsilon\ox \epsilon\ox B\ox A}\\
(B\ox A)^{\ox 2}\ox B^{\ox 2}
\ar[d]_-{\psiinv\ox B\ox A\ox B^{\ox 2}}&
(A\ox B)^{\ox 2}\ar[dd]^-{A\ox \epsilon\ox A\ox B}&
A\ox B^{\ox 3}\ox A\ar[r]^-{A\ox \mu\ox B\ox A}&
A\ox B^{\ox 2}\ox A\ar[rr]^-{\epsilon\ox \epsilon\ox B\ox A}&&
B\ox A\ar@{=}[dd]&
B\ox A\ar[l]_-{\psi\psiinv}\ar@{->>}[dd]\\
(A\ox B^{\ox 2})^{\ox 2}\ar[d]_-{(A\ox \mu)^{\ox 2}}&&&&&&\\
(A\ox B)^{\ox 2}\ar[r]_-{\quad A\ox \epsilon\ox A\ox B}&
A^{\ox 2}\ox B\ar[rr]_-{\epsilon\ox A\ox B}&&
A\ox B\ar[rr]_-\psi&&
B\ox A&
\ \ B\stac \psi A\ar@{>->}[l]
}
\end{equation}

\eject
\begin{equation}\label{eq:antipright}
\scalebox{.82}{
{\color{white}.}\hspace{-1.1cm}
\xymatrix @R=15pt @C=2pt{
B\stac\psi A 
\ar[r]^-{\raisebox{8pt}{${}_{(B\ox_\psi A)\ox \eta}$}}&
(B\stac \psi A)\!\ox\! A\ar[rrrrrr]^-{(B\ox_\psi A)\ox A\ox \eta}&&&&&&
(B\stac \psi A)\!\ox\! A\!\ox\! B
\ar[r]^-{\raisebox{8pt}{${}_{(B\ox_\psi A)\ox \psi}$}}&
(B\stac \psi A)\!\ox\! B\!\ox\! A
\ar[r]^-{\raisebox{8pt}{${}_{(B\ox_\psi A)\ox \Delta_{B\ox A}}$}}&
(B\stac \psi A)\!\ox\! (B\!\ox\! A)^{\ox 2}\ar@{->>}[r]&
(B\stac \psi A)^{\ox 3}\ar@{>->}[d]\\
B\!\ox\! A \ar@{->>}[u]\ar[r]^-{B\ox A\ox \eta}\ar[d]^-{\psi\psiinv}&
B\!\ox\! A^{\ox 2}\ar[rrrrrr]^-{B\ox A^{\ox 2}\ox \eta}
\ar[d]^-{\psi\psiinv\ox A}&&&&&&
B\!\ox\! A^{\ox 2}\!\ox\! B\ar[r]^-{B\ox A \ox \psi}&
(B\!\ox\! A)^{\ox 2}
\ar[r]^-{\raisebox{5pt}{${}_{B\ox A \ox \Delta_{B\ox A}}$}}&
(B\!\ox\! A)^{\ox 3}\ar@{->>}[u]\ar[r]^-{(\psi\psiinv)^{\ox 3}}
\ar[d]^-{\psi\psiinv\ox (B\ox A)^{\ox 2}}&
(B\!\ox\! A)^{\ox 3}\ar[dd]_-{\flip\ox B\ox A}\\
B\!\ox\! A \ar[dd]^-{\Delta\ox A}&
(B\!\ox\! A)\!\ox\! (A)\ar[rrrrrr]^-{B\ox A^{\ox 2}\ox \eta}
\ar[dd]^-{\flip}\ar@{}[rrrrrrddd]|-{\red \eqref{eq:antip1}}&&&&&&
B\!\ox\! A^{\ox 2}\!\ox\! B\ar[r]^-{B\ox A\ox \psi}
\ar[d]_-{B\ox A\ox \Delta_{A\ox B}}
\ar@{}[rrd]|-{\red \eqref{eq:psibar}}&
(B\!\ox\! A)^{\ox 2}
\ar[r]^-{\raisebox{5pt}{${}_{B\ox A\ox \Delta_{B\ox A}}$}}&
(B\!\ox\! A)^{\ox 3}\ar[d]^-{(B\ox A)^{\ox 2}\ox \psi\psiinv}&\\
&&&&&&&
(B\!\ox\! A)\!\ox\! (A\!\ox\! B)^{\ox 2}\ar[rr]^-{B\ox A\ox \psi^{\ox 2}}
\ar[d]_-{\flip\ox A\ox B}&&
(B\!\ox\! A)^{\ox 3}\ar[d]^-{\flip\ox B\ox A}&
(B\!\ox\! A)^{\ox 3}\ar[dd]_-{B\ox \psi\ox A\ox B\ox A}\\
B^{\ox 2}\!\ox\! A\ar[dd]^(.3){B\ox \flip}&
A\!\ox\! B\!\ox\! A\ar[d]^(.4){A\ox \Delta\ox A}&&&&&&
A\!\ox\! B^{\ox 2}\!\ox\! A^{\ox 2}\!\ox\! B
\ar[r]^-{\raisebox{6pt}{${}_{A\ox B^{\ox 2}\ox A\ox \psi}$}}
\ar[d]_-{A\ox \mu\ox A^{\ox 2}\ox B}&
A\!\ox\! B\!\ox\! (B\!\ox\! A)^{\ox 2}
\ar[r]^-{\raisebox{8pt}{${}_{\psi\ox (B\ox A)^{\ox 2}}$}}
\ar[ldd]^-{A\ox \mu \ox A\ox B\ox A}
\ar@{}[rdd]|(.3){\red \eqref{eq:wdl}}&
(B\!\ox\! A)^{\ox 3}\ar[d]^-{B\ox \psi\ox A\ox B\ox A}&\\
&
A\!\ox\! B^{\ox 2}\!\ox\! A
\ar[r]^-{\raisebox{5pt}{${}_{A\ox B\ox S\ox A}$}}&
A\!\ox\! B^{\ox 2}\!\ox\! A\ar[rr]^-{\Delta_{A\ox B}\ox B\ox A}
\ar[ddd]_-{\psi\ox B\ox A}\ar@{}[rrddd]|-{\red \eqref{eq:psibar}}&&
A\!\ox\! B\!\ox\! (A\!\ox\! B^{\ox 2})\!\ox\! (A)
\ar[r]^-{\raisebox{6pt}{${}_{A\ox B\ox \flip}$}}
\ar[d]^(.4){A\ox B\ox \psi\ox B\ox A}&
A\!\ox\! B \!\ox\! A^{\ox 2}\!\ox\! B^{\ox 2}
\ar[rr]^-{A\ox B\ox A^{\ox 2}\ox \mu}
\ar[d]^(.4){A\ox B\ox A\ox \psi\ox B}
\ar@{}[rrd]|-{\red \eqref{eq:wdl}}&&
A\!\ox\! B\!\ox\! A^{\ox 2}\!\ox\! B\ar[d]_(.4){A\ox B\ox A\ox \psi}&&
B^{\ox 2}\!\ox\! A^{\ox 2}\!\ox\! B\!\ox\! A
\ar[d]^-{\mu\ox A^{\ox 2}\ox B\ox A}&
B^{\ox 2}\!\ox\! A^{\ox 2}\!\ox\! B\!\ox\! A
\ar[dd]_-{\mu \ox \mu\ox B\ox A}\\
B\!\ox\! A\!\ox\! B\ar[ddd]^-{B\ox \sqcap^L\ox B}&&&&
A\!\ox\! B \!\ox\!(B\!\ox\! A)^{\ox 2}
\ar[d]^(.4){\psi\ox (B\ox A)^{\ox 2}}&
(A\!\ox\! B)^{\ox 3}
\ar[r]^-{\raisebox{6pt}{${}_{(A\ox B)^{\ox 2}\ox \psi}$}}&
(A\!\ox\! B)^{\ox 2}\!\ox\! B\!\ox\! A
\ar[r]^-{\raisebox{5pt}{${}_{A\ox B\ox A\ox \mu\ox A}$}}&
(A\!\ox\! B)^{\ox 2}\!\ox\! A\ar[rr]^-{\psi\ox A\ox B\ox A}&&
B\!\ox\! A^{\ox 2}\!\ox\! B\!\ox\! A\ar@{=}[d]&\\
&&&&
B\!\ox\! A\!\ox\! (B\!\ox\! A\!\ox\! B)\!\ox\! (A)
\ar[r]^-{\raisebox{6pt}{${}_{B\ox A\ox \flip}$}}
\ar@{}[rrrrrd]|-{\red \eqref{eq:onright}}&
B\!\ox\! A\!\ox (A\!\ox\! B)^{\ox 2}
\ar[rr]^-{B\ox A^{\ox 2}\ox B\ox \psi}&&
B\!\ox\! A^{\ox 2}\!\ox\! B^{\ox 2}\!\ox\! A
\ar[rr]^-{B\ox A^{\ox 2}\ox \mu\ox A}&&
B\!\ox\! A^{\ox 2}\!\ox\! B\!\ox\! A \ar@{=}[d]
\ar@{}[rd]|-{\red \eqref{eq:onleft} \eqref{eq:onright}}&
(B\!\ox\! A)^{\ox 2}\ar@{->>}[d]\\
&&
(B\!\ox\! A)^{\ox 2}\ar[rr]^-{\Delta_{B\ox A}\ox B\ox A}
\ar[rd]^-{\ B\ox \Delta\ox B\ox A}
\ar[d]_-{B\ox A\ox B\ox \sqcap^L}&&
B\!\ox\! A\!\ox\! (B\!\ox\! A\!\ox\! B)\!\ox\! (A)
\ar[u]_-{B\ox A\ox \psi\psiinv \ox B\ox A}
\ar[r]^-{\raisebox{6pt}{${}_{B\ox A\ox \flip}$}}&
B\!\ox\! A\!\ox\! (A\!\ox\! B)^{\ox 2}
\ar[rr]^-{B\ox A^{\ox 2}\ox B\ox \psi}
\ar@{}[rrrrddd]|-{\red\mathrm{(counit)}}&&
B\!\ox\! A^{\ox 2}\!\ox\! B^{\ox 2}\!\ox\! A
\ar[rr]^-{B\ox A^{\ox 2}\ox \mu\ox A}&&
B\!\ox\! A^{\ox 2}\!\ox\! B\!\ox\! A\ar[d]^-{B\ox \mu\ox B\ox A}&
(B\stac \psi A)^{\ox 2}\ar@{>->}[d]\\
B\!\ox\! A\!\ox\! B \ar[dd]^-{B\ox A\ox S}&&
(B\!\ox\! A)^{\ox 2}\ar[dd]_-{B\ox A\ox \flip}&
B\!\ox\! A^{\ox 2}\!\ox\! B\!\ox\! A
\ar[r]^-{\raisebox{6pt}{${}_{\flip\ox A\ox B\ox A}$}}
\ar@{}[rrdd]|-{\red\eqref{eq:antip2}}&
A\!\ox\! (B\!\ox\! A\!\ox\! B)\!\ox\! (A)\ar[r]^-{A\ox \flip}&
A^{\ox 2}\!\ox\! B\!\ox\! A\!\ox\! B\ar[d]^-{\mu\ox B\ox A\ox B}&&&&
(B\!\ox\! A)^{\ox 2}\ar[dd]^(.3){\epsilon\ox \epsilon \ox B\ox A}&
(B\!\ox\! A)^{\ox 2}\ar[dd]_-{\epsilon\ox \epsilon \ox B\ox A}\\
&&&&&
(A\!\ox\! B)^{\ox 2} \ar[d]^-{\epsilon\ox B\ox A\ox B}&&&&&\\
B\!\ox\! A\!\ox\! B
\ar[r]^-{\raisebox{6pt}{${}_{\eta\ox B\ox A\ox B}$}}\ar@{=}[d]
\ar@{}[rrrrrrrrrd]|-{\red\eqref{eq:onright}}&
(A\!\ox\! B)^{\ox 2}
\ar[r]^-{\raisebox{6pt}{${}_{\psi\ox A\ox B}$}}&
B\!\ox\! A^{\ox 2}\!\ox\! B\ar[rrr]^-{B\ox \mu\ox B}&&&
B\!\ox\! A\!\ox\! B \ar[rr]^-{B\ox \psi}&&
B^{\ox 2}\!\ox\! A\ar[rr]^-{\mu\ox A}&&
B\!\ox\! A\ar@{=}[d]&
B\!\ox\! A\ar[l]_-{\psi\psiinv}\ar@{->>}[d]\\
B\!\ox\! A\!\ox\! B
\ar[rrrrr]_-{B\ox \psi}&&&&&
B^{\ox 2}\!\ox\! A\ar[rrrr]_-{\mu \ox A}&&&&
B\!\ox\! A&
\ B\stac\psi A \ar@{>->}[l]
}}
\end{equation}

\eject
{\small
\begin{equation}\label{eq:antipleft}
\xymatrix{
B\stac\psi A \ \ar@{>->}[r]&
B\ox A \ar[rr]^-{\Delta_{B\ox A}}
\ar@{}[rrd]|-{\red \eqref{eq:psi}}&&
(B\ox A)^{\ox 2}\ar@{->>}[r]\ar[rd]^-{(\psi\psiinv)^{\ox 2}}&
(B\stac \psi A)^{\ox 2}\ar@{>->}[d]\\
B\ox A\ar@{->>}[u]\ar[ru]^-{\psi\psiinv}\ar[r]^-{\psi\psiinv}&
B\ox A\ar[rr]^-{\Delta_{B\ox A}}\ar[d]_-{\Delta \ox A}&&
(B\ox A)^{\ox 2}\ar[r]^-{\psi\psiinv\ox B\ox A}
\ar[dd]^-{B\ox A\ox \flip}&
(B\ox A)^{\ox 2}\ar[dd]^-{B\ox A\ox \flip}\\
&B^{\ox 2}\ox A \ar[d]_-{B\ox \flip}&&&\\
&B\ox A\ox B\ar[rr]^-{B\ox \Delta \ox B}\ar[d]_-{B\ox \sqcap^L \ox B}&&
B\ox A^{\ox 2}\ox B\ar[d]^-{B\ox A\ox S\ox S}&
B\ox A^{\ox 2}\ox B\ar[d]^-{B\ox A\ox S\ox S}\\
&B\ox A\ox B \ar[dd]_-{B\ox A\ox S}
\ar@{}[rd]|-{\red \eqref{eq:antip_ax}}&
B\ox A^{\ox 2}\ox B\ar[dd]^-{B\ox \mu\ox B}\ar@{=}[r]
\ar@{}[rdddd]|-{\red \eqref{eq:wdl}}&
B\ox A^{\ox 2}\ox B\ar[d]^-{B\ox A \ox \psi}&
B\ox A^{\ox 2}\ox B\ar[d]^-{B\ox A \ox \psi}\\
&&&(B\ox A)^{\ox 2}\ar[r]^-{\psi\psiinv\ox B\ox A}
\ar[dd]^-{B\ox \psi\ox A}
\ar@{}[rddddd]|-{\red \eqref{eq:onleft}\eqref{eq:onright}}&
(B\ox A)^{\ox 2}\ar@{->>}[d]\\
&B\ox A\ox B\ar@{=}[r]&
B\ox A\ox B \ar[dd]^-{B\ox \psi}&&
(B\stac \psi A)^{\ox 2}\ar@{>->}[d]\\
&&&B^{\ox 2}\ox A^{\ox 2}\ar[d]^-{B^{\ox 2}\ox \mu}&
(B\ox A)^{\ox 2}\ar[d]^-{B\ox \psi\ox A}\\
&&B^{\ox 2}\ox A\ar@{=}[r]&
B^{\ox 2}\ox A\ar[dd]^-{\mu\ox A}&
B^{\ox 2}\ox A^{\ox 2}\ar[d]^-{\mu\ox \mu}\\
&&&&B\ox A\ar@{->>}[d]\ar[ld]_-{\psi\psiinv}\\
&&&B\ox A&
\ B\stac \psi A\ar@{>->}[l]
}
\end{equation}
}
\vfill 
\eject

\end{landscape}

\end{document}